\documentclass[12pt,a4paper,reqno]{amsart}
\usepackage{amsmath}
\usepackage{amsthm}
\usepackage{bbold}
\usepackage{amsfonts}
\usepackage{amssymb}
\usepackage{float}
\usepackage{amsthm}
\usepackage{psfrag}
\usepackage{enumerate}
\usepackage{hyperref}

\newcommand{\rsdraw}[3]{\raisebox{-#1\height}{\scalebox{#2}{\includegraphics{#3.eps}}}}
\newtheorem{thm}{Theorem}
\newtheorem{lem}{Lemma}
\newtheorem{cor}{Corollary}
\newcommand{\co}{\colon}

\newcommand{\tens}{\otimes}
\newcommand{\id}{\mathrm{id}}

\newcommand{\fs}[1]{[#1]}

\newcommand{\mc}{\mathcal}
\newcommand{\opp}{\mathrm{op}}
\newcommand{\N}{\mathbb{N}}
\newcommand{\R}{\mathbb{R}}

\newcommand{\Hom}{\mathrm{Hom}}
\newcommand{\End}{\text{End}}

\newcommand{\Ob}[1]{\mathrm{Ob}(#1)}

\newcommand{\RSL}{\textbf{RSL}}

\newcommand{\uu}{\Bbb1}
\newcommand{\evl}{\text{ev}}
\newcommand{\evr}{\widetilde{\text{ev}}}
\newcommand{\coevl}{\text{coev}}
\newcommand{\coevr}{\widetilde{\text{coev}}}
\newcommand{\Reid}[1]{\textbf{R}#1}
\newcommand{\coend}{\mathbb{F}}

\usepackage{geometry}
\makeatletter
\@namedef{subjclassname@2020}{\textup{2020} Mathematics Subject Classification}
\makeatother

\usepackage{etoolbox}
\makeatletter
\patchcmd{\@setaddresses}{\indent}{\noindent}{}{}
\patchcmd{\@setaddresses}{\indent}{\noindent}{}{}
\patchcmd{\@setaddresses}{\indent}{\noindent}{}{}
\patchcmd{\@setaddresses}{\indent}{\noindent}{}{}
\makeatother
\newlength{\gnat}

\begin{document}
	\title{Cyclic sets from ribbon string links}
	\author{Ivan Bartulovi\'c}
	
	\address{Universit\'{e} de Lille, Laboratoire Paul Painlev\'{e}
		UMR CNRS 8524, F-59000 Lille, France}
	\email{ivan.bartulovic@univ-lille.fr}

	\begin{abstract} In this paper, we first endow the set of ribbon string links (up to isotopy) with a structure of a cyclic and of a cocyclic set. Next, we relate these (co)cyclic sets with those associated with the coend of a ribbon category. The relationship is given by the universal quantum invariants à la Reshetikhin-Turaev.
		
		
	\end{abstract}
	\maketitle
	\setcounter{tocdepth}{1}
	\tableofcontents
	\section{Introduction}

A (co)cyclic object in a category is, roughly speaking, a (co)simplicial object with compatible actions of the cyclic groups. In particular, a (co)cyclic set/vector space is a (co)cyclic object in the category of sets/vector spaces. The first example arose in homological algebra: Connes~\cite{Connesext,connes_non-commutative_1985} associated to any algebra a cocyclic vector space whose cohomology is called the cyclic cohomology of the algebra.  Majid and Akrami \cite{cycliccocycles} gene\-ralized this construction by associating to any ribbon algebra (an algebra in a braided monoidal category equipped with a ribbon automorphism) a cocyclic vector space.

In this paper, we first endow the set of ribbon string links (which are framed long knots with several components) with the structure of a cyclic set and of a cocyclic set (see Theorem~\ref{slmain} and Section~\ref{SLL}). Next, we prove that these (co)cyclic sets universally ``dominate'' the (co)cyclic sets associated to the coends of ribbon categories (see Theorem~\ref{slmain2} and Corollary~\ref{corslmain2}).  To be more specific, consider the category~\textbf{RSL} of ribbon string links (where composition is given by stacking).  Fol\-lowing Bruguières-Virelizier~\cite{hopfdiag}, the quantum invariants à la Reshetikhin-Turaev associated to a ribbon category $\mc B$ give rise to a functor~$\phi_{\mc B}$ from~\textbf{RSL} to the  convolution category~$\text{Conv}_{\mc B}(\coend,\uu)$, where~$\coend$ is the coend of $\mc B$ (endowed with its canonical coalgebra structure) and $\uu$ is the monoidal unit of $\mc B$ (endowed with the trivial algebra structure). We show that the functor $\phi_{{\mc B}}$ induces a morphism of (co)cyclic sets from the (co)cyclic sets associated to ribbon string links to the (co)cyclic sets (à la Akrami-Majid \cite{cycliccocycles} and its variants, see Section \ref{algebraiccyclic}) associated with the braided Hopf algebra~$\coend$.

	The paper is organized as follows.
	In Section~\ref{cycprelimini}, we recall the notions of (co)simp\-li\-cial and (co)cyclic objects in a category. In Section~\ref{slrelatedobjects}, we construct (co)cyclic sets from ribbon string links.
	In Section~\ref{catprelimini}, we review ribbon categories and graphical calculus.
	Section~\ref{algebraiccyclic} is dedicated to (co)cyclic objects from categorical (co)algebras.
	In Section~\ref{relquantum}, we relate, via the quantum invariants, the (co)cyclic sets from string links to those associated to a coend of a ribbon category.
	Throughout the paper, the class of objects of a category~$\mc C$ is denoted by~$\Ob{\mc C}$.

\subsection*{Acknowledgments} This work was supported by the Labex CEMPI (ANR-11-LABX-0007-01), by the Région Hauts-de-France, and by the FNS-ANR OChoTop grant (ANR-18-CE93-0002-01). The author is particularly grateful for the valuable comments and advice of his thesis supervisor Alexis Virelizier.

	\section{Cyclic objects} \label{cycprelimini}
	In this section we review the cyclic category, which is an extension of the simplicial category and is used to define (co)cyclic objects in a category.
	
	\subsection{The simplicial category}
	The \emph{simplicial category}~$\Delta$ is defined as follows.
	The objects of~$\Delta$ are the non-negative integers~$n \in \N$. For $n\in \N$, consider the ordered sets~$\fs{n}=\{0,\dots,n\}$.
	A morphism from~$n$ to $m$ in~$\Delta$ is an increasing map~$\fs{n}\to\fs{m}$.
	For~$n\in \N^*$ and~$0\le i \le n$, the $i$-th \emph{coface}~$\delta_i^n \co (n-1) \to n$ is the unique increasing injection from~$\fs{n-1}$ into~$\fs{n}$ which misses~$i$.
	For~$n\in \N$ and~$0\le j \le n$, the~$j$-th~\emph{codegene\-racy}~$\sigma_j^n \co (n+1) \to n$ is the unique increasing surjection from~$\fs{n+1}$ onto~$\fs{n}$ which sends both~$j$ and~$j+1$ to~$j$.
	
	It is well known (see~\cite[Lemma 5.1]{MCLHomology}) that morphisms in~$\Delta$ are generated by the cofaces~$\{\delta_i^n\}_{n\in \N^*, 0\leq i \leq n}$ and the codegeneracies~$\{\sigma_j^n\}_{n\in \N, 0\leq j \leq n}$ subject to the following three \textit{simplicial relations}:
	\begin{align}
	\label{cofaces} \delta_j^{n+1} \delta_i^n &= \delta_i^{n+1} \delta_{j-1}^n \text{\quad for all }0 \le i<j \le n+1, \\
	\label{codegeneracies}  \sigma_j^{n}\sigma_i^{n+1}&=\sigma_i^n\sigma_{j+1}^{n+1} \text{\quad for all } 0 \le i \leq j \le n,
	\\
	\label{compcofcod}\sigma_j^n\delta_i^{n+1}&= \begin{cases}
	\delta_i^n \sigma_{j-1}^{n-1} & \text{for all } 0 \le i<j \le n, \\
	\id_{n} & \text{for } 0\le i=j \le n \text{ or } 1\le i=j+1 \le n+1,\\
	\delta_{i-1}^n\sigma_j^{n-1} & \text{for all } 1\le j+1<i\le n+1.
	\end{cases}
	\end{align}
	In the op\-po\-site ca\-te\-go\-ry~$\Delta^{\opp}$, the co\-face~$\delta_i^n$ and the co\-de\-ge\-ne\-ra\-cy~$\sigma_j^n$ are re\-spec\-tively de\-noted by
	\[
	d_i^n \co n \to (n-1) \quad \text{and} \quad s_j^n \co n \to (n+1).
	\]
	The morphisms $\{d_i^n\}_{n\in \N^*, 0\le i \le n}$ are called the \textit{faces} and the morphisms $\{s_j^n\}_{n\in \N, 0 \le j \le n}$ are called the \textit{degeneracies}.
	
	\subsection{The cyclic category}
	The~\emph{cyclic category}~$\Delta C$ (introduced by Connes \cite{Connesext}) is defined as follows. The objects of~$\Delta C$ are the non-negative integers~$n\in \N$.
	The morphisms are generated by the morphisms~$\{\delta_i^n\}_{n\in \N^*, 0\le i \le n}$, called \textit{cofaces}, the morphisms~$\{\sigma_j^n\}_{n\in \N, 0\le j \le n}$, called \textit{codegeneracies}, and the isomorphisms ~$\{\tau_n\co n \to n\}_{n\in \N}$, called~\emph{cocyclic operators}, satisfying the simplicial relations~{\eqref{cofaces}-\eqref{compcofcod}}, the following four \textit{compatibility relations}:
	\begin{align}
	\label{compcoccof}&\tau_n \delta_i^n = \delta_{i-1}^n \tau_{n-1} \text{\quad for all } 1 \leq i \leq n, \\
	\label{tndelta0}&\tau_n\delta_0^n=\delta_n^n, \text{\quad for all } n \ge 1, \\
	\label{compcoccod}&\tau_n\sigma_i^n=\sigma_{i-1}^n\tau_{n+1} \text{\quad for all } 1 \leq i \leq n,\\
	\label{tnsigma0}&\tau_n\sigma_0^n=\sigma_n^n\tau_{n+1}^2, \text{\quad for all } n \ge 0,
	\end{align}
	and the \textit{cocyclicity condition}
	\begin{equation}\label{cocyclicity}
	\tau_n ^{n+1} = \id_{n} \quad \text{for all} \quad n \in \N.
	\end{equation}
	In the op\-po\-site ca\-te\-go\-ry~$\Delta C^{\opp}$, the co\-face~$\delta_i^n$, the co\-de\-ge\-ne\-ra\-cy~$\sigma_j^n$, and the cocyclic ope\-rator $\tau_n$ are respectively de\-noted by
	\[
	d_i^n \co n \to (n-1), \quad s_j^n \co n \to (n+1), \quad \text{and} \quad t_n \co n \to n.
	\]
	The morphisms $\{d_i^n\}_{n\in \N^*, 0\le i \le n}$ are called the \textit{faces}, the morphisms $\{s_j^n\}_{n\in \N, 0 \le j \le n}$ are called the \textit{degeneracies}, and the morphisms $\{t_n\}_{n\in \N}$ are called the \textit{cyclic operators}.

	\subsection{(Co)simplicial and (co)cyclic objects in a category} \label{cycobj}
	Let~$\mc C$ be a category.
	A~\textit{simplicial object} in~$\mc C$ is a functor~$\Delta^\opp \to \mc C$. A~\emph{cyclic object} in~$\mc C$ is a functor~$\Delta C^{\opp} \to \mc C.$
	Dually, a~\emph{cosimplicial object} in~$\mc{C}$ is a functor~$\Delta \to \mc C$ and a~\emph{cocyclic object} in~$\mc C$ is a functor~$\Delta C \to \mc C$. A \textit{morphism} between two (co)simplicial/(co)cyclic objects is a natu\-ral transformation between them. A (co)simplicial/(co)cyclic object in the cate\-gory of sets are called \textit{(co)simplicial/(co)cyclic sets}.
	
	Since the categories~$\Delta$ and $\Delta C$ are defined by gene\-rators and relations, a (co)sim\-pli\-cial/(co)cyclic object in a category is completely determined by the ima\-ges of the gene\-rators satisfying the corresponding relations. As usual, we denote these images by the same letter. For example, a cocyclic object~$X$ in~$\mc C$ may be described explicitly by a fami\-ly~$X^\bullet=\{X_n\}_{n\in \N}$ of objects in~$\mc C$ and by morphisms~$\{\delta_i^n\co X^{n-1} \to X^{n}\}_{n\in \N^*,0\leq i \leq n}$, called~\textit{cofaces}, morphisms~$\{\sigma_j^n\co X^{n+1} \to X^n\}_{n\in \N, 0 \leq j \leq n}$, called~\textit{codegeneracies}, and isomorphisms~$\{\tau_n\co X^n \to X^n\}_{n\in \N}$, called~\textit{cocyclic operators}, which satisfy~\eqref{cofaces}-\eqref{cocyclicity}.
	From this point of view, a morphism~$\alpha^\bullet \co X^\bullet \to Y^\bullet$ between cocyclic objects~$X^\bullet$ and~$Y^\bullet$ in~$\mc C$ is described by a fami\-ly~$\alpha^\bullet=\{\alpha^n\co X^n \to Y^n\}_{n\in \N}$ of morphisms in~$\mc C$ such that
	\begin{align*}
	\delta_i^n\alpha^{n-1}&= \alpha^n\delta_i^n \quad \text{for all } n\ge 1 \text{ and } 0\leq i \leq n, \\
	\sigma_j^n \alpha^{n+1}&= \alpha^n\sigma_j^n \quad \text{for all } n\ge 0 \text{ and } 0\leq j \leq n, \\
	\alpha^n\tau_n&=\tau_n\alpha^n \quad \text{for all } n \ge 0.
	\end{align*}
Similarly as above, a cyclic object $X$ in $\mc C$ is described by a family $X_\bullet=\{X_n\}_{n\in \N}$ of objects in $\mc C$ and by morphisms~$\{d_i^n\co X_n \to X_{n-1}\}_{n\in \N^*,0\leq i \leq n}$, called~\textit{faces}, morphisms~$\{s_j^n\co X_n \to X_{n+1}\}_{n\in \N, 0\leq j \leq n}$, called~\textit{degeneracies}, and isomorphisms~$\{t_n\co X_n \to X_n\}_{n\in \N}$, called \textit{cyclic operators}, which satisfy the relations
	\begin{align}
	\label{faces} d_i^nd_j^{n+1} &= d_{j-1}^nd_i^{n+1} \text{\quad for all } 0\le i<j \le n+1, \\
	\label{degeneracies}s_i^{n+1}s_j^n&=s_{j+1}^{n+1}s_i^n \text{\quad for all } 0\le i \leq j\le n, \\
	\label{compfacdeg} d_i^{n+1}s_j^n&= \begin{cases}
	s_{j-1}^{n-1}d_i^n & \text{for all } 0\le i<j \le n, \\
	\id_{X_n} & \text{for } 0\le i=j \le n \text{ or } 1\le i=j+1 \le n+1,\\
	s_j^{n-1}d_{i-1}^n & \text{for all } 1\le j+1<i\le n+1,
	\end{cases} \\
	\label{compcycfac} d_i^nt_n &= t_{n-1}d_{i-1}^n \text{\quad for all } 1 \leq i \leq n,\\
	\label{d0tn} d_0^nt_n&=d_n^n \text{\quad for all } n\ge 1, \\
	\label{compcycdeg}s_i^nt_n&=t_{n+1}s_{i-1}^n \text{\quad for all } 1 \leq i \leq n,\\
	\label{s0tn} s_0^nt_n&=t_{n+1}^2s_n^n \text{\quad for all } n\ge 0, \hspace{0.7cm}
	\end{align}
	and such that the~\textit{cyclicity condition} holds for any~$n\in \N$:
	\begin{equation}\label{cyclicity}
	t_n ^{n+1} = \id_{X_n}.
	\end{equation}
	In this characterization, a morphism $\alpha_\bullet \co X_\bullet \to Y_\bullet$ between two cyclic objects~$X_\bullet$ and~$Y_\bullet$ in~$\mc C$ is described by a family~$\alpha_\bullet=\{\alpha_n\co X_n \to Y_n\}_{n\in \N}$ of morphisms in~$\mc C$ satisfying
	\begin{align*}
	\alpha_{n-1}d_i^n&=d_i^n\alpha_n \quad \text{for all } n\ge 1 \text{ and } 0\leq i \leq n, \\
	\alpha_{n+1}s_j^n&=s_j^n\alpha_n \quad \text{for all } n\ge 0 \text{ and } 0\leq j \leq n, \\
	\alpha_nt_n&=t_n\alpha_n \quad \text{for all }  n\ge 0.
	\end{align*}
	
	\subsection{Cyclic duality} \label{connes loday dual}
	
	Connes defined in \cite{Connesext} an isomorphism of categories~$\Delta C \cong \Delta C^\opp$ called \textit{cyclic duality}.
	This cyclic duality~$L \co \Delta C^\opp \to \Delta C$ (in its version due to Loday \cite[Chapter 6]{loday98}) is identity on objects and it is defined on morphisms as follows.
	For~$n\geq 1$ and~$0\leq i \leq n$,
	\[L(d_i^n)=
	\begin{cases}
	\sigma_i^{n-1} & \text{if } 0\leq i \leq n-1, \\
	\sigma_0^{n-1}\tau_{n}^{-1} &  \text{if } i=n,
	\end{cases}\]
	and for $n\geq 0$ and $0\leq j \leq n$,
	\[L(s_j^n)=\delta_{j+1}^{n+1} \quad \text{and} \quad L(t_n)=\tau_n^{-1}.\]
	Given a category $\mc C$, the cyclic duality transforms a cocyclic object $X \co \Delta C \to \mc C$ into the cyclic object $XL \co \Delta C^\opp \to \mc C$ and its opposite $L^\opp$ transforms a cyclic object $Y \co \Delta C^\opp \to \mc C$ into the cocyclic object $Y L^\opp \co \Delta C \to \mc C$.

	\section{Cyclic and cocyclic sets from ribbon string links}\label{slrelatedobjects}
	In this section, we construct (co)cyclic sets from ribbon string links.
	
	\subsection{Ribbon string links}\label{sldef}
	
	A~\textit{ribbon} is a homeomorphic image of the rectangle~$[0,1]\times [0,1]$.
	The image of the segment~$[0,1]\times \{0\}$ is called the~\textit{bottom base} and the image of the segment~$[0,1]\times \{1\}$ is called the~\textit{top base} of the ribbon.
	The image of the segment~$\{\frac{1}{2}\}\times [0,1]$ is called the~\textit{core} of the ribbon.
	Let~$n$ be a non-negative integer. A \textit{ribbon~$n$-string link} is an oriented
	surface~$T$ embedded in the strip~$\R^2\times [0, 1]$ and decomposed into a disjoint union of $n$ ribbons such that~$T$ meets the planes~$\R^2\times \{0\}$ and~$\R^2\times \{1\}$ orthogonally as follows. For all~$1 \le k \le n$, the bottom base and the top base of the~$k$-th ribbon of~$T$ are respectively the segments $\left[k-\frac{1}{4}, k+\frac{1}{4}\right]\times \{0\} \times \{0\}$ and~$\left[k-\frac{1}{4}, k+\frac{1}{4}\right]\times \{0\} \times \{1\}$, and in the points of these segments, the orientation of $T$ is determined by the pair of vectors~$(1, 0, 0)$ and~$(0, 0, 1)$ tangent to $T$.

Note that there is a unique ribbon $0$-string link, which is the empty set.
 By an \textit{isotopy of ribbon string links}, we mean isotopy in~$\R^2 \times [0, 1]$ constant on the boundary and preserving the splitting into ribbons as well as the orientation of the surface~$T$.

We represent a ribbon string link~$T$ by a plane diagram with the \textit{blackboard framing} convention: the ribbons of~$T$ should go close and parallel to the plane~$\R \times \{0\} \times \R$ and the orientation of~$T$ corresponds to the counterclockwise orientation in~$\R \times \{0\} \times \R$.
We represent then~$T$ by the projection of the cores of its ribbons onto the plane~$\R \times \{0\} \times \R$ so that there are only double
transversal crossings (with overcrossing and undercrossing information).
For example,
$$\rsdraw{0.4}{0.44}{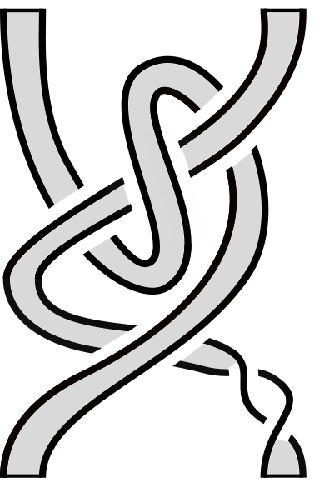}\quad  \text{is isotopic to} \quad  \rsdraw{0.4}{0.44}{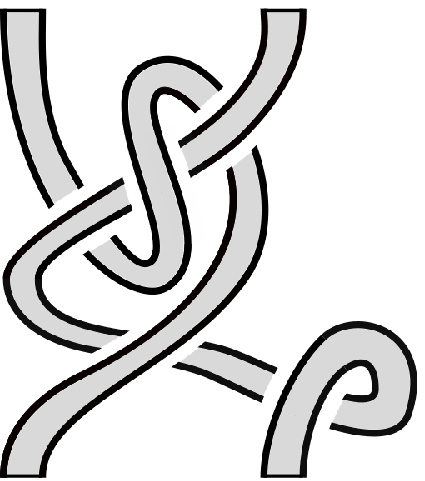}, \quad \text{and is then represented by the diagram} \quad  \rsdraw{0.4}{0.88}{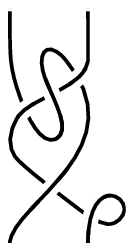}.$$

	Any diagram defines a ribbon string link (up to isotopy). Two planar diagrams represent isotopic string links if and only if they are related by a finite sequence of planar isotopies fixing the bases and the following ribbon Reidemeister moves:
	$$
	\settowidth{\gnat}{\rsdraw{0.5}{1}{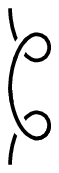}\;$\leftrightarrow$\; \rsdraw{0.5}{1}{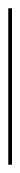}\;$\leftrightarrow$\rsdraw{0.5}{1}{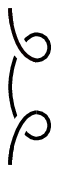}\;,}\begin{minipage}[c]{\gnat}
	\begin{center} \rsdraw{0.5}{1}{RI1}\;$\leftrightarrow$\; \rsdraw{0.5}{1}{RI3}\;$\leftrightarrow$\rsdraw{0.5}{1}{RI2}\;, $$\Reid{1}$$
	\end{center}
	\end{minipage} \qquad \settowidth{\gnat}{\rsdraw{0.5}{1}{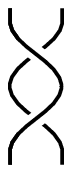}\;$\leftrightarrow$ \rsdraw{0.5}{1}{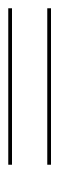}\;$\leftrightarrow$\rsdraw{0.5}{1}{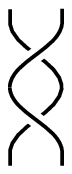}\;,}\begin{minipage}[c]{\gnat}
	\begin{center} \rsdraw{0.5}{1}{RII1}\;$\leftrightarrow$ \rsdraw{0.5}{1}{RII2}\;$\leftrightarrow$\rsdraw{0.5}{1}{RII3}\;, $$\Reid{2}$$
	\end{center}
	\end{minipage} \qquad \settowidth{\gnat}{\rsdraw{0.5}{1}{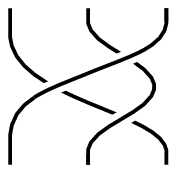}\;$\leftrightarrow$ \rsdraw{0.5}{1}{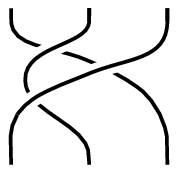}\;.}\begin{minipage}[c]{\gnat}
	\begin{center} \rsdraw{0.5}{1}{RIII1}\;$\leftrightarrow$ \rsdraw{0.5}{1}{RIII2}\;. $$\Reid{3}$$
	\end{center}
	\end{minipage}$$

	The category~$\RSL$ of ribbon string links has as objects non-negative integers.
	For any two non-negative integers $m$ and $n$, the set of morphisms from $m$ to $n$ is defined by  \[
	\Hom_{\RSL}(m,n)=
	\begin{cases}
	\text{isotopy classes of ribbon $n$-string links} \quad &\text{if } m=n, \\
	\emptyset \quad& \text{if } m\neq n.
	\end{cases}
	\]
	The composition $T'\circ T$ of two ribbon~$n$-string links is given by stacking~$T'$ on the
	top of~$T$ (i.e., with ascending convention) and compressing the result into $\R^2\times \left[0,1\right]$:
	$$
	T' \circ T =
	\,
	\psfrag{T}[cc][cc][0.85]{$T$}
	\psfrag{S}[cc][cc][0.85]{$T'$}
	\rsdraw{0.45}{1}{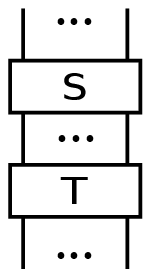}\;.
	$$
Identity of $n$ is the trivial ribbon $n$-string link
$$\id_n= 	\,
\psfrag{b}{\hspace{-0.07cm}$1$}
\psfrag{e}{$n$}
\rsdraw{0.55}{1}{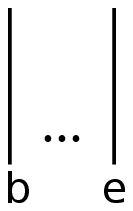}\; $$
As above, we number the ribbons of a ribbon string link from the left to the right.
	For any~$n\in \N$, we denote by~$\mc{RSL}_n$ the monoid $\End_{\textbf{RSL}}(n+1)$ of the isotopy classes of ribbon~$(n+1)$-string links.
	
	\subsection{(Co)cyclic sets from ribbon string links}
	For any~$n\in \N$, define~$\mc{SL}^n=\mc{RSL}_n$ as a set. Next, define the cofaces~$\{\delta_i^n \co \mc{SL}^{n-1} \to \mc{SL}^n\}_{n\in \N^*, 0\le i \le n}$, the codegeneracies $\{\sigma_j^n \co \mc{SL}^{n+1}\to \mc{SL}^n\}_{n\in \N, 0\le j \le n}$, and the cocyclic operators~$\{\tau_n\co \mc{SL}^n \to \mc{SL}^n\}_{n\in \N}$ by setting
	\begin{gather*}
	\delta_0^n(T)=\,
	\psfrag{b}[cc][cc][0.75]{$1$}
	\psfrag{e}[cc][cc][0.75]{$n$}
	\psfrag{T}[cc][cc][0.85]{$T$}
	\rsdraw{0.55}{1}{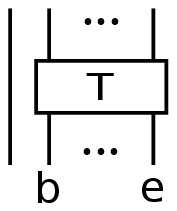} \;, \quad \delta_i^n(T)=\,
	\psfrag{b}[cc][cc][0.75]{$1$}
	\psfrag{i}[cc][cc][0.75]{\hspace{0.3cm}$i$}
	\psfrag{i+1}[cc][cc][0.75]{\hspace{0.3cm}$i+1$}
	\psfrag{e}[cc][cc][0.75]{$n$}
	\psfrag{T}[cc][cc][0.85]{$T$}
	\rsdraw{0.55}{1}{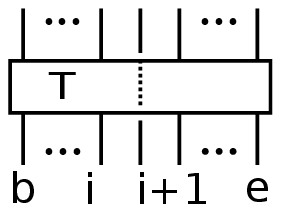}\;, \quad \delta_n^n(T) = \,
	\psfrag{b}[cc][cc][0.75]{$1$}
	\psfrag{e}[cc][cc][0.75]{$n$}
	\psfrag{T}[cc][cc][0.85]{$T$}
	\rsdraw{0.55}{1}{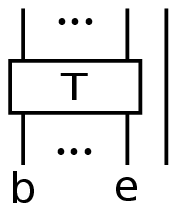}\;, \\ \sigma_j^n(T)=\,
	\psfrag{b}[cc][cc][0.75]{$0$}
	\psfrag{j}[cc][cc][0.75]{$j$}
	\psfrag{j+1}[cc][cc][0.75]{\hspace{0.3cm}$j+1$}
	\psfrag{e}[cc][cc][0.75]{\hspace{0.4cm}$n+1$}
	\psfrag{T}[cc][cc][0.85]{$T$}
	\rsdraw{0.55}{1}{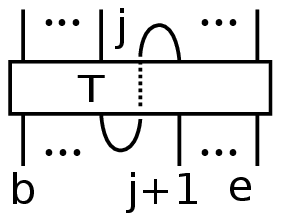}\;, \quad \tau_0(T)=T, \quad \tau_n(T)=\,
	\psfrag{b}[cc][cc][0.75]{$0$}
	\psfrag{pe}[cc][cc][0.75]{$n-1$}
	\psfrag{e}[cc][cc][0.75]{$n$}
	\psfrag{T}[cc][cc][0.75]{$T$}
	\rsdraw{0.55}{1}{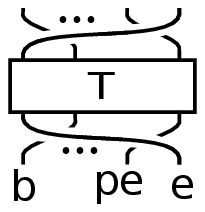}
	\;.
	\end{gather*}
	The string link $\delta_i^n(T)$ is obtained from $T$ by inserting from behind a trivial component between the $i$-th and $(i+1)$-th component. The string link~$\sigma_j^n(T)$ is obtained from~$T$ by connecting from behind the~$(j+1)$-th and~$(j+2)$-th component.
	
	Similarly as above, for any $n\in \N$, define  $\mc{SL}_n=\mc{RSL}_n$ as a set.
	Next, define the faces $\{d_i^n \co \mc{SL}_{n}  \to \mc{SL}_{n-1} \}_{n\in \N^*, 0\le i \le n}$, the degeneracies $\{s_j^n\co \mc{SL}_n\to \mc{SL}_{n+1}\}_{n\in \N, 0\le j \le n}$, and the cyclic operators $\{t_n\co \mc{SL}_n\to \mc{SL}_n\}_{n\in \N}$ by setting
	\begin{gather*}
	d_i^n(T)=
	\,
	\psfrag{b}[cc][cc][0.75]{$0$}
	\psfrag{i}[cc][cc][0.75]{$i$}
	\psfrag{e}[cc][cc][0.75]{$n$}
	\psfrag{T}[cc][cc][0.85]{$T$}
	\rsdraw{0.55}{1}{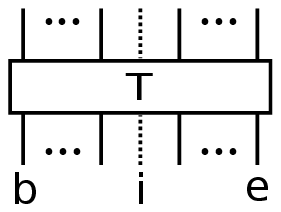}\;, \hspace{0.25cm}  s_j^n(T)=\,
	\psfrag{b}[cc][cc][0.75]{$0$}
	\psfrag{j}[cc][cc][0.75]{\hspace{0.25cm}$j$}
	\psfrag{e}[cc][cc][0.75]{$n$}
	\psfrag{T}[cc][cc][0.85]{$T$}
	\rsdraw{0.55}{1}{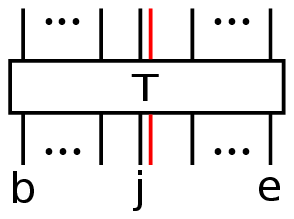}\;, \hspace{0.25cm}
	t_0(T)=T, \hspace{0.25cm} t_n (T)=
	\,
	\psfrag{b}[cc][cc][0.75]{$0$}
	\psfrag{nb}[cc][cc][0.75]{\hspace{-0.4cm}$1$}
	\psfrag{e}[cc][cc][0.75]{\hspace{0.12cm}$n$}
	\psfrag{T}[cc][cc][0.85]{$T$}
	\rsdraw{0.55}{1}{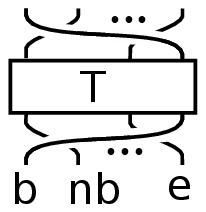}
	\;.
	\end{gather*}
	The string link~$d_i^n(T)$ is obtained from~$T$ by deleting the~$(i+1)$-th component. The string link~$s_j^n(T)$ is obtained from~$T$ by duplicating, along the framing, the~$(j+1)$-th component. Note that the removal and duplication operations for string links appeared in the work of Habiro~\cite{habiro2006bottom}.
	
	\begin{thm} \label{slmain}
		\begin{itemize}
			\item[(a)] The fa\-mi\-ly~$\mc{SL}^\bullet=\{\mc{SL}^n\}_{n\in \N}$ endowed with the co\-fa\-ces~$\{\delta_i^n \}_{n\in \N^*, 0\le i \le n}$, the co\-de\-ge\-ne\-ra\-cies~$\{\sigma_j^n\}_{n\in \N, 0\le j \le n}$ and the co\-cyc\-lic o\-pe\-ra\-tors~$\{\tau_n\}_{n\in \N}$
			is a cocyclic set.
			\item[(b)] The fa\-mi\-ly~$\mc{SL}_\bullet=\{\mc{SL}_n\}_{n\in \N}$ en\-dowed with the fa\-ces~$\{d_i^n\}_{n \in \N^*, 0\le i \le n}$, the de\-ge\-ne\-ra\-cies~$\{s_j^n\}_{n\in \N, 0\le j \le n}$ and the cyc\-lic o\-pe\-ra\-tors~$\{t_n\}_{n\in \N}$ is a cyc\-lic set.
		\end{itemize}
	\end{thm}
	We prove Theorem \ref{slmain} in Section \ref{proofslmain}.

	\subsection{Cyclic duals}
	\label{SLL}
	
 	By precomposing the cyclic duality~$L\co \Delta C^\opp \to \Delta C$ from Section~\ref{connes loday dual} with the cocyclic set~${\mc {SL}}^\bullet$ from Theorem~\ref{slmain}$(a)$, we obtain the cyclic set~${\mc {SL}}^\bullet\circ L$.
	By definitions,~$\mc{SL}^\bullet\circ L(n)= \mc{SL}^n= \mc{RSL}_n$ for all~$n\in \N$.
	The faces $\{\tilde{d}_i\co \mc{RSL}_n \to \mc{RSL}_{n-1}\}_{n\in \N^*, 0 \le i \le n}$, the degeneracies $\{\tilde{s}_j^n\co \mc{RSL}_n\to \mc{RSL}_{n+1}\}_{n\in \N, 0\le j \le n}$, and the cyclic operators $\{\tilde{t}_n\co \mc{RSL}_n\to \mc{RSL}_n\}_{n\in \N}$ of the cyclic set~${\mc {SL}}^\bullet\circ L$ are computed by
	\begin{gather*}
	\tilde{d}_i^n(T)=\,
	\psfrag{b}[cc][cc][0.75]{$0$}
	\psfrag{j}[cc][cc][0.75]{$i$}
	\psfrag{j+1}[cc][cc][0.75]{\hspace{0.5cm}$i+1$}
	\psfrag{e}[cc][cc][0.75]{\hspace{0.5cm}$n$}
	\psfrag{T}[cc][cc][0.85]{$T$}
	\rsdraw{0.55}{1}{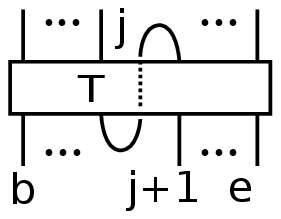} \;, \quad
	\tilde{d}_n^n(T)=\,
	\psfrag{b}[cc][cc][0.75]{\hspace{-0.2cm}$0$}
	\psfrag{i}[cc][cc][0.75]{$1$}
	\psfrag{be}[cc][cc][0.75]{\hspace{-0.75cm}$n-1$}
	\psfrag{e}[cc][cc][0.75]{$n$}
	\psfrag{T}[cc][cc][0.85]{$T$}
	\rsdraw{0.55}{1}{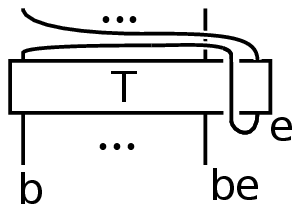} \;, \quad  \tilde{s}_j^n(T) =
	\,
	\psfrag{b}[cc][cc][0.75]{$0$}
	\psfrag{i}[cc][cc][0.75]{\hspace{0.25cm}$j$}
	\psfrag{i+1}[cc][cc][0.75]{\hspace{0.2cm}$j+1$}
	\psfrag{e}[cc][cc][0.75]{$n$}
	\psfrag{T}[cc][cc][0.85]{$T$}
	\rsdraw{0.55}{1}{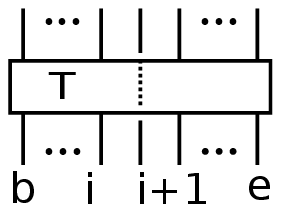}\;, \\
	\tilde{s}_n^n(T) =
	\,
	\psfrag{b}[cc][cc][0.75]{$0$}
	\psfrag{e}[cc][cc][0.75]{$n$}
	\psfrag{T}[cc][cc][0.85]{$T$}
	\rsdraw{0.55}{1}{SLdeltan}\;, \quad \tilde{t}_0(T)=T, \quad
	\tilde{t}_n(T)=
	\,
	\psfrag{b}[cc][cc][0.75]{$0$}
	\psfrag{nb}[cc][cc][0.75]{\hspace{-0.3cm}$1$}
	\psfrag{e}[cc][cc][0.75]{\hspace{0.15cm}$n$}
	\psfrag{T}[cc][cc][0.85]{$T$}
	\rsdraw{0.55}{1}{SLcycneg}
	\;.
	\end{gather*}
	
	Similarly as above, by precomposing the functor~$L^\opp \co \Delta C \to \Delta C^\opp$ with the cyclic set~${\mc {SL}}_\bullet$ from Theorem  \ref{slmain}$(b)$, we obtain the cocyclic set~${\mc {SL}}_\bullet\circ L^\opp$. It follows by definitions that~$\mc{SL}_\bullet\circ L^\opp(n)=\mc{SL}_n=\mc{RSL}_n$ for all~$n\in \N$. The cofaces~$\{\tilde{\delta}_i^n\co \mc{RSL}_{n-1} \to \mc{RSL}_{n}\}_{n\in \N^*, 0 \le i \le n}$, the codegeneracies~$\{\tilde{\sigma}_j^n\co \mc{RSL}_{n+1} \to \mc{RSL}_{n}\}_{n\in \N, 0 \le j \le n}$, and the cocyclic operators $\{\tilde{\tau}_n\co \mc{RSL}_n\to \mc{RSL}_n\}_{n\in \N}$ of the cocyclic set~${\mc {SL}}_\bullet \circ L^{\opp}$ are computed by
	\begin{gather*}
	\tilde{\delta}_i^n(T)=
	\,
	\psfrag{b}[cc][cc][0.75]{$1$}
	\psfrag{j}[cc][cc][0.75]{\hspace{0.26cm}$i+1$}
	\psfrag{e}[cc][cc][0.75]{$n$}
	\psfrag{T}[cc][cc][0.85]{$T$}
	\rsdraw{0.55}{1}{SLduplication} \;, \quad
	\tilde{\delta}_n^n(T)=
	\,
	\psfrag{b}[cc][cc][0.75]{$1$}
	\psfrag{be}[cc][cc][0.75]{\hspace{0.5cm}$n-1$}
	\psfrag{e}[cc][cc][0.75]{$n$}
	\psfrag{T}[cc][cc][0.85]{$T$}
	\rsdraw{0.55}{1}{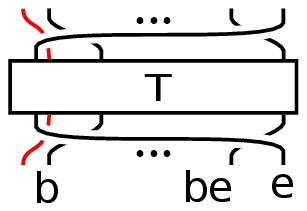} \;, \quad
	\tilde{\sigma}_j(T)=\,
	\psfrag{b}[cc][cc][0.75]{$0$}
	\psfrag{i}[cc][cc][0.75]{$j+1$}
	\psfrag{e}[cc][cc][0.75]{$n+1$}
	\psfrag{T}[cc][cc][0.85]{$T$}
	\rsdraw{0.55}{1}{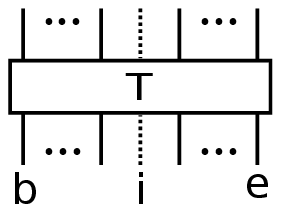}\;, \\
	\tilde{\tau}_0(T)=T, \quad \tilde{\tau}_n(T)=
	\,
	\psfrag{b}[cc][cc][0.75]{$0$}
	\psfrag{pe}[cc][cc][0.75]{$n-1$}
	\psfrag{e}[cc][cc][0.75]{\hspace{0.14cm}$n$}
	\psfrag{T}[cc][cc][0.85]{$T$}
	\rsdraw{0.55}{1}{SLcyc}
	\;.
	\end{gather*}

	\subsection{Proof of Theorem \ref{slmain}} \label{proofslmain}
We only prove the part $(a)$ of the theorem, by checking the relations \eqref{cofaces}-\eqref{cocyclicity}. The proof of the part $(b)$ is similar and is left to the reader (one needs to verify relations~\eqref{faces}-\eqref{cyclicity}).

Let us prove the part $(a)$. We first verify~\eqref{cofaces}.
 If~$1\le i<j \le n-1$ and~$T\in \mc{RSL}_{n-1}$, then
	\begin{align*}
	\delta_j^{n+1}\delta_i^n(T)&\overset{(i)}{=}\delta_j^{n+1}\left(
	\,
	\psfrag{b}[cc][cc][0.75]{$1$}
	\psfrag{i}[cc][cc][0.75]{\hspace{0.3cm}$i$}
	\psfrag{i+1}[cc][cc][0.75]{\hspace{0.3cm}$i+1$}
	\psfrag{e}[cc][cc][0.75]{$n$}
	\psfrag{T}[cc][cc][0.85]{$T$}
	\rsdraw{0.55}{1}{SLdeltai}\;\right)
	\overset{(ii)}{=}
	\,
	\psfrag{b}[cc][cc][0.75]{$1$}
	\psfrag{i}[cc][cc][0.75]{$i$}
	\psfrag{i+1}[cc][cc][0.75]{\hspace{0.3cm}$i+1$}
	\psfrag{j}[cc][cc][0.75]{$j$}
	\psfrag{j-1}[cc][cc][0.75]{\hspace{0.4cm}$j-1$}
	\psfrag{e}[cc][cc][0.75]{$n$}
	\psfrag{T}[cc][cc][0.85]{$T$}
	\rsdraw{0.55}{1}{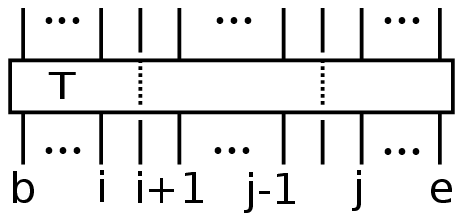}\;
	\\&\overset{(iii)}{=}\delta_i^{n+1}\left(
	\,
	\psfrag{b}[cc][cc][0.75]{$1$}
	\psfrag{i}[cc][cc][0.75]{$j-1$}
	\psfrag{i+1}[cc][cc][0.75]{\hspace{0.2cm}$j$}
	\psfrag{e}[cc][cc][0.75]{$n$}
	\psfrag{T}[cc][cc][0.85]{$T$}
	\rsdraw{0.55}{1}{SLdeltai}\;\right)
	\overset{(iv)}{=}\delta_i^{n+1}\delta_{j-1}^n(T).
	\end{align*}
	Here $(i)$, $(iii)$, $(iv)$ follow from the definition and $(ii)$ follows from the definition and the hypothesis that $i<j$.
	Indeed, since we count the unlabeled trivial component, which is inserted between the components labeled by $i$ and $i+1$, the $j$-th component of $\delta_i^n(T)$ is the one labeled by $j-1$ on the string link $T$.
	The cases $i=0<j\le n+1$ and $i=n, j=n+1$ are checked in a similar way.
	
	Next, we verify~\eqref{codegeneracies}.
	Let~$n\ge 0$.
	If~$i<j$ and~$T\in \mc{RSL}_{n+2}$, then
	\begin{align*}
	\sigma_j^n\sigma_i^{n+1}(T)&\overset{(i)}{=}\sigma_j^n\left(\,
	\psfrag{b}[cc][cc][0.75]{$0$}
	\psfrag{j}[cc][cc][0.75]{$i$}
	\psfrag{j+1}[cc][cc][0.75]{\hspace{0.5cm}$i+1$}
	\psfrag{e}[cc][cc][0.75]{\hspace{0.5cm}$n+2$}
	\psfrag{T}{$T$}
	\rsdraw{0.55}{1}{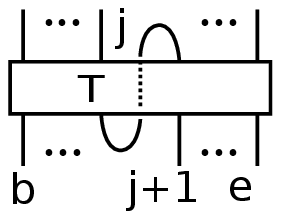}\;\quad\right)\overset{(ii)}{=}
	\,
	\psfrag{b}[cc][cc][0.75]{$0$}
	\psfrag{i}[cc][cc][0.75]{$i$}
	\psfrag{i+1}[cc][cc][0.75]{\hspace{0.5cm}$i+1$}
	\psfrag{j+1}[cc][cc][0.75]{$j+1$}
	\psfrag{j+2}[cc][cc][0.75]{\hspace{0.7cm}$j+2$}
	\psfrag{e}[cc][cc][0.75]{$n+2$}
	\psfrag{T}{$T$}
	\rsdraw{0.45}{1}{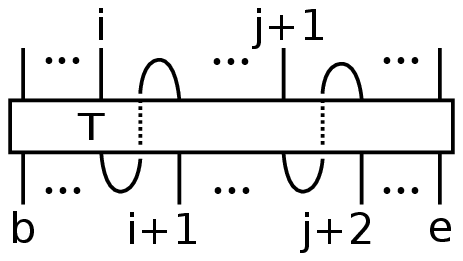}\; \\
	&\overset{(iii)}{=}\sigma_{i}^n\left(\,
	\psfrag{b}[cc][cc][0.75]{$0$}
	\psfrag{j}[cc][cc][0.75]{\hspace{0.3cm}$j+1$}
	\psfrag{j+1}[cc][cc][0.75]{\hspace{0.6cm}$j+2$}
	\psfrag{e}[cc][cc][0.75]{\hspace{0.6cm}$n+2$}
	\psfrag{T}[cc][cc][0.85]{$T$}
	\rsdraw{0.45}{1}{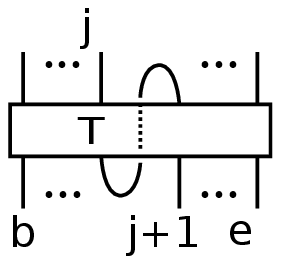}\;\quad\right)\overset{(iv)}{=}\sigma_i^n\sigma_{j+1}^{n+1}(T).
	\end{align*}
	Here $(i),(iii),(iv)$ follow from the definition and $(ii)$ follows from the definition and the hypothesis that $i<j$.
	Indeed, since one concatenates the components labeled by~$i$ and~$i+1$, the~$(j+1)$-th component of~$\sigma_i^{n+1}(T)$ is the one labeled by~$j+1$ on the string link~$T$. The case~$i=j$ is trivial to check.
	
	Let us verify relations \eqref{compcofcod}. Let $T \in \mc{RSL}_{n}$. If $i=j$ and $i\neq 0$, we have
	\[\sigma_i^n\delta_i^{n+1}(T)\overset{(i)}{=}\sigma_i^n\left(
	\,
	\psfrag{b}[cc][cc][0.75]{$1$}
	\psfrag{i}[cc][cc][0.75]{\hspace{0.3cm}$i$}
	\psfrag{i+1}[cc][cc][0.75]{\hspace{0.3cm}$i+1$}
	\psfrag{e}[cc][cc][0.75]{\hspace{0.1cm}$n+1$}
	\psfrag{T}[cc][cc][0.85]{$T$}
	\rsdraw{0.55}{1}{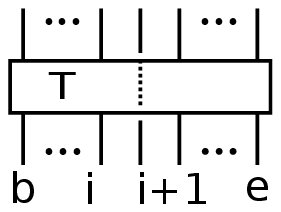}\; \hspace{0.2cm}\right)\overset{(ii)}{=}\,
	\psfrag{b}[cc][cc][0.75]{$0$}
	\psfrag{i}[cc][cc][0.75]{$i-1$}
	\psfrag{i+1}[cc][cc][0.75]{\hspace{0.22cm}$i+1$}
	\psfrag{e}[cc][cc][0.75]{$n+1$}
	\psfrag{T}[cc][cc][0.85]{$T$}
	\rsdraw{0.55}{1}{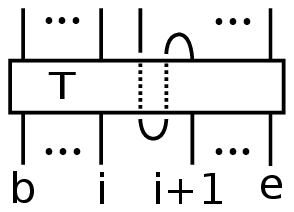}\;\overset{(iii)}{=}T.\]
	Here~$(i)$ follows from the definition, $(iii)$ follows by the isotopy, and~$(ii)$ follows from the definition and since the~$(i+1)$-th component of the string link~$\delta_i^{n+1}(T)$ is the unlabeled component inserted between the components labeled by~$i$ and~$i+1$ on the string link $T$.
	The case~$i=j=0$ is trivial to check. Next, consider the case when $i<j$. If $i\neq 0$, we have
	\begin{align*}
	\sigma_j^n\delta_i^{n+1}(T)&\overset{(i)}{=}\sigma_j^n\left(\,
	\psfrag{b}[cc][cc][0.75]{$1$}
	\psfrag{i}[cc][cc][0.75]{\hspace{0.3cm}$i$}
	\psfrag{i+1}[cc][cc][0.75]{\hspace{0.22cm}$i+1$}
	\psfrag{e}[cc][cc][0.75]{\hspace{0.1cm}$n+1$}
	\psfrag{T}[cc][cc][0.85]{$T$}
	\rsdraw{0.55}{1}{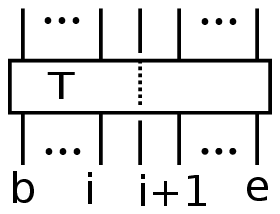}\;\hspace{0.2cm}\right)\overset{(ii)}{=}
	\,
	\psfrag{b}[cc][cc][0.75]{$1$}
	\psfrag{i}[cc][cc][0.75]{\hspace{0.3cm}$i$}
	\psfrag{i+1}[cc][cc][0.75]{\hspace{0.3cm}$i+1$}
	\psfrag{j}[cc][cc][0.75]{$j$}
	\psfrag{j+1}[cc][cc][0.75]{\hspace{0.62cm}$j+1$}
	\psfrag{e}[cc][cc][0.75]{\hspace{0.05cm}$n+1$}
	\psfrag{T}[cc][cc][0.85]{$T$}
	\rsdraw{0.45}{1}{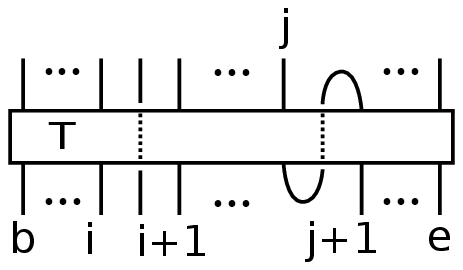}\;\\&\overset{(iii)}{=}\delta_i^   {n}\left(\,
	\psfrag{b}[cc][cc][0.75]{$1$}
	\psfrag{j}[cc][cc][0.75]{$j$}
	\psfrag{j+1}[cc][cc][0.75]{\hspace{0.4cm}$j+1$}
	\psfrag{e}[cc][cc][0.75]{\hspace{0.6cm}$n+1$}
	\psfrag{T}[cc][cc][0.85]{$T$}
	\rsdraw{0.55}{1}{SLsigmaj}\;
\hspace{0.2cm}
	\right)\overset{(iv)}{=}\delta_i^n\sigma_{j-1}^{n-1}(T).
	\end{align*}
	Here $(i)$, $(iii)$, $(iv)$ follow from the definition and $(ii)$ follows from the definition and the hypothesis that $i<j$.
	Indeed, since we count the unlabeled trivial component, which is inserted between the components labeled by $i$ and $i+1$, the $(j+1)$-th component of the string link $\delta_i^{n+1}(T)$ is the one labeled by $j$ on the string link $T$.
	We proceed in the same way if~$i=0$. The cases when $i>j+1$ or $i=j+1$ are proven analogously.

	Let us verify the relation~\eqref{compcoccof}. Assume that~$2\le i \le n-1$. For~$T\in \mc{RSL}_{n-1}$, we have
	\begingroup
	\allowdisplaybreaks
	\begin{align*}
	\tau_n\delta_i^n(T)&\overset{(i)}{=}\,
	\psfrag{b}[cc][cc][0.75]{$1$}
	\psfrag{nb}[cc][cc][0.75]{$2$}
	\psfrag{i-1}[cc][cc][0.75]{\hspace{-0.4cm}$i-1$}
	\psfrag{i}[cc][cc][0.75]{\hspace{-0.2cm}$i$}
	\psfrag{i+1}[cc][cc][0.75]{\hspace{-0.6cm}$i+1$}
	\psfrag{e}[cc][cc][0.75]{$n$}
	\psfrag{T}[cc][cc][0.85]{$T$}
	\rsdraw{0.47}{1}{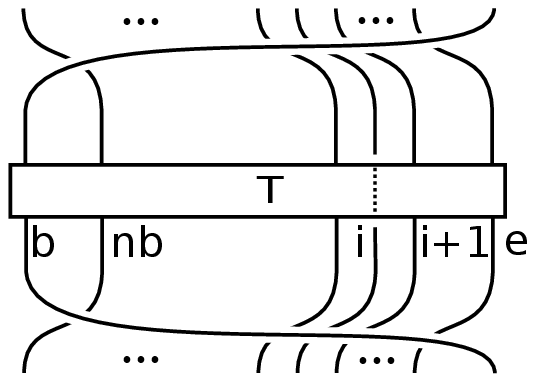}\;\overset{(ii)}{=}\,
	\psfrag{b}[cc][cc][0.75]{$1$}
	\psfrag{nb}[cc][cc][0.75]{$2$}
	\psfrag{i-1}[cc][cc][0.75]{\hspace{-0.1cm}$i-1$}
	\psfrag{i}[cc][cc][0.75]{\hspace{-0.4cm}$i$}
	\psfrag{i+1}[cc][cc][0.75]{\hspace{-0.75cm}$i+1$}
	\psfrag{e}[cc][cc][0.75]{$n$}
	\psfrag{T}[cc][cc][0.85]{$T$}
	\rsdraw{0.47}{1}{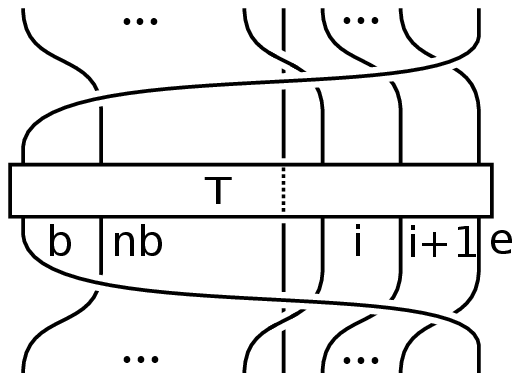}\;  \\&\\
	&\overset{(iii)}{=}\delta_{i-1}^n\left(\,
	\psfrag{b}[cc][cc][0.75]{$1$}
	\psfrag{nb}[cc][cc][0.75]{$2$}
	\psfrag{i-1}[cc][cc][0.75]{$i-1$}
	\psfrag{i}[cc][cc][0.75]{\hspace{-0.4cm}$i$}
	\psfrag{i+1}[cc][cc][0.75]{\hspace{-0.15cm}$i+1$}
	\psfrag{e}[cc][cc][0.75]{$n$}
	\psfrag{T}[cc][cc][0.85]{$T$}
	\rsdraw{0.47}{1}{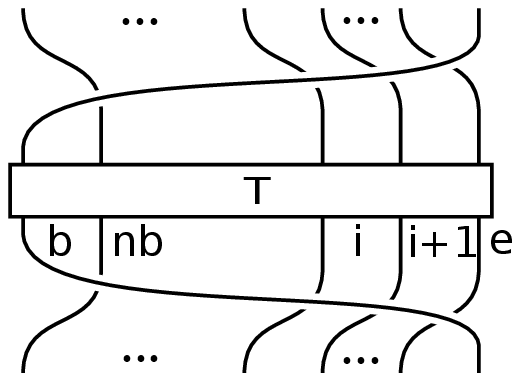}\;\right)\overset{(iv)}{=}\delta_{i-1}^n\tau_{n-1}(T).
	\end{align*}
	\endgroup
	Here $(i), (iii), (iv)$ follow from the definition and $(ii)$ follows by isotopy and \Reid{3} move. The remaining cases are shown in the same manner.
	
	Let us check the relation \eqref{compcoccod}.
	If $1\le i \le n$ and $T\in \mc {RSL}_{n+1}$, then
	\begin{align*}
	\tau_n\sigma_i^n(T)&\overset{(i)}{=}\,
	\psfrag{b}[cc][cc][0.75]{$0$}
	\psfrag{j}[cc][cc][0.75]{$i$}
	\psfrag{j+1}[cc][cc][0.75]{\hspace{-0.5cm}$i+1$}
	\psfrag{e}[cc][cc][0.75]{\hspace{0.6cm}$n+1$}
	\psfrag{T}[cc][cc][0.85]{$T$}
	\rsdraw{0.45}{1}{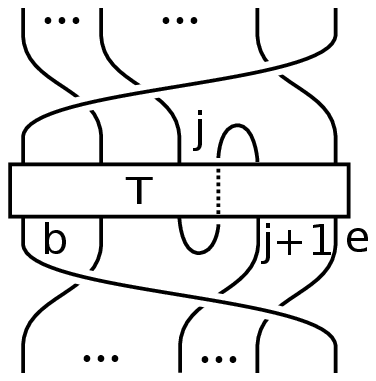}\;\hspace{1.3cm}\overset{(ii)}{=}\,
	\psfrag{b}[cc][cc][0.75]{$0$}
	\psfrag{j}[cc][cc][0.75]{$i$}
	\psfrag{j+1}[cc][cc][0.75]{\hspace{-0.5cm}$i+1$}
	\psfrag{e}[cc][cc][0.75]{\hspace{0.6cm}$n+1$}
	\psfrag{T}[cc][cc][0.85]{$T$}
	\rsdraw{0.45}{1}{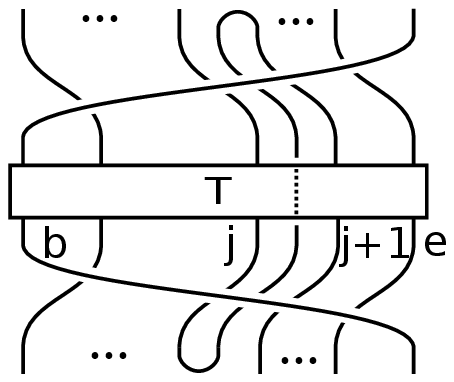}\;\\&\\&\overset{(iii)}{=}\,
	\psfrag{b}[cc][cc][0.75]{$0$}
	\psfrag{j}[cc][cc][0.75]{$i$}
	\psfrag{j+1}[cc][cc][0.75]{\hspace{-0.5cm}$i+1$}
	\psfrag{e}[cc][cc][0.75]{\hspace{0.5cm}$n+1$}
	\psfrag{T}[cc][cc][0.85]{$T$}
	\rsdraw{0.45}{1}{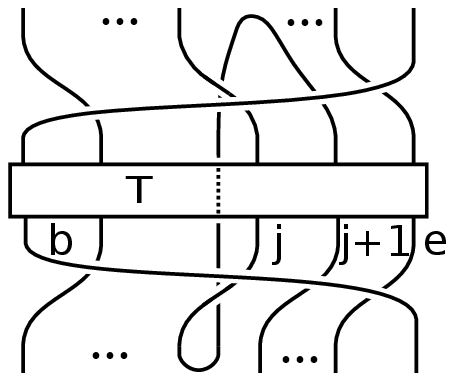}\;\hspace{0.5cm}\overset{(iv)}{=}\sigma_{i-1}^n\left(\,
	\psfrag{b}[cc][cc][0.75]{$0$}
	\psfrag{j-1}[cc][cc][0.75]{\hspace{-0.65cm}$i-1$}
	\psfrag{j}[cc][cc][0.75]{\hspace{-0.7cm}$i$}
	\psfrag{e}[cc][cc][0.75]{\hspace{-0.1cm}$n+1$}
	\psfrag{T}[cc][cc][0.85]{$T$}
	\rsdraw{0.45}{1}{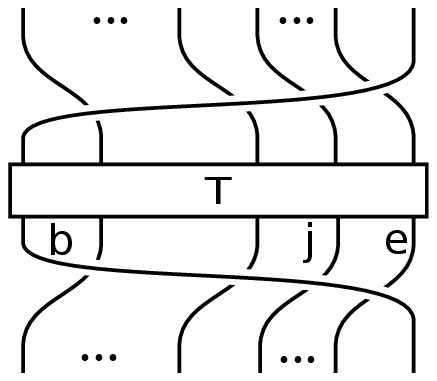}\;\hspace{0.2cm}\right)\\&\overset{(v)}{=}\sigma_{i-1}^n\tau_{n+1}(T).
	\end{align*}
	Here~$(i),(iv),(v)$ follow from the definition,~$(ii)$ follows by isotopy, and $(iii)$ follows by isotopy and~\Reid{3} move.
	
	According to \cite[Section 6.1.1]{loday98}, the relation~\eqref{tndelta0} is a consequence of relations~\eqref{cocyclicity} and~\eqref{compcoccof}.
	Similarly, the relation~\eqref{tnsigma0} is a consequence of relations~\eqref{cocyclicity} and~\eqref{compcoccod}.
	Hence, it suffices to show that the relation~\eqref{cocyclicity} holds.
	We show it in the case~$n=1$.
	The general case is treated similarly.
	If~$T \in \mc{RSL}_1$, then
	\[
	\tau_1^2(T)\overset{(i)}{=}
	\,
	\psfrag{b}[cc][cc][0.75]{\hspace{-0.1cm}$0$}
	\psfrag{e}[cc][cc][0.75]{$1$}
	\psfrag{T}[cc][cc][0.85]{$T$}
	\rsdraw{0.45}{1}{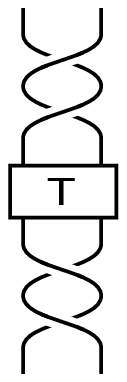} \; \overset{(ii)}{=}
	\,
	\psfrag{b}[cc][cc][0.75]{\hspace{-0.1cm}$0$}
	\psfrag{e}[cc][cc][0.75]{$1$}
	\psfrag{T}[cc][cc][0.85]{$T$}
	\rsdraw{0.45}{1}{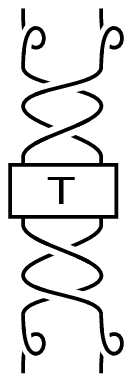}\;\overset{(iii)}{=}\,
	\psfrag{b}[cc][cc][0.75]{\hspace{-0.1cm}$0$}
	\psfrag{e}[cc][cc][0.75]{$1$}
	\psfrag{T}[cc][cc][0.85]{$T$}
	\rsdraw{0.45}{1}{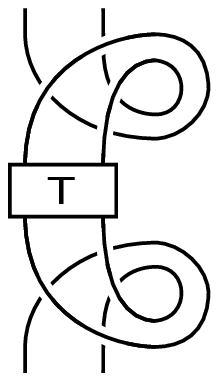}\;\overset{(iv)}{=}\,
	\psfrag{b}[cc][cc][0.75]{\hspace{-0.1cm}$0$}
	\psfrag{e}[cc][cc][0.75]{$1$}
	\psfrag{T}[cc][cc][0.85]{$T$}
	\rsdraw{0.45}{1}{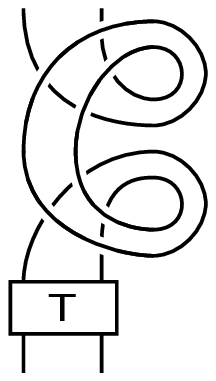}\;\overset{(v)}{=}T.
	\]
	Here~$(i)$ follows from the definition,~$(ii)$ by adding one po\-si\-tive and one ne\-ga\-tive left hand twist on each component and by using the na\-tu\-ra\-li\-ty~of twists,~$(iii)$ by iso\-to\-py and~$\Reid{3}$~move,~$(iv)$ by isotopy,~$(v)$ by isotopy,~$\Reid{2}$ move, and~$\Reid{3}$ move. The general case is treated similarly.
	
	\section{Ribbon categories and graphical calculus}\label{catprelimini}
	
	In this section, we recall some algebraic preliminaries on ribbon categories and their graphical calculus used in the remaining sections of the paper. For more details, see \cite{moncatstft}.
	
	\subsection{Conventions} In what follows, we suppress in our formulas the associativity and unitality constraints of the monoidal category. We denote by $\tens$ and $\uu$ the monoidal product and unit object of a monoidal category. For any objects $X_1, \dots, X_n$ of a monoidal category with $n\ge 2$, we set $$X_1 \tens X_2 \tens \cdots \tens X_n=(\cdots ((X_1\tens X_2)\tens X_3)\tens \cdots \tens X_{n-1} )\tens X_n$$ and similarly for morphisms.

	\subsection{Braided categories} \label{braidedcats}
	A \textit{braiding} of a monoidal category $(\mc B, \tens, \uu)$ is a family~$\tau=\{\tau_{X,Y}\co X\tens Y \to Y\tens X\}_{X,Y\in \Ob{\mc B}}$ of natural isomorphisms such that
	\begin{align*}
	\tau_{X,Y\tens Z}&= (\id_Y \tens \tau_{X,Z})(\tau_{X,Y} \tens \id_Z) \text{ and } \\
	\tau_{X\tens Y, Z}&= (\tau_{X,Z} \tens \id_Y)(\id_X \tens \tau_{Y,Z})
	\end{align*}
	for all~$X,Y,Z \in \Ob{\mc B}$.
	A \textit{braided category} is a monoidal category endowed with a braiding.
	A braiding~$\tau$ of~$\mc B$ is \textit{symmetric} if for all $X,Y \in \Ob{\mc B}$, \[\tau_{Y,X}\tau_{X,Y}=\id_{X\tens Y}.\]
	A \textit{symmetric category} is a monoidal category endowed with a symmetric braiding.

	\subsection{Braided categories with a twist} \label{catswithtwist} A \textit{twist} for a braided monoidal category~$\mc B$ is a natural isomorphism $\theta=\{\theta_X \co X \to X\}_{X \in \Ob{\mc B}}$ such that for all $X,Y \in \Ob{\mc B}$,
	\begin{equation} \label{twistcondition}\theta_{X\tens Y}= \tau_{Y,X}\tau_{X,Y} (\theta_X \tens \theta_Y). \end{equation}
	Note that \eqref{twistcondition} implies $\theta_\uu= \id_\uu$.
	For example, the family $\id_{\mc B}=\{\id_X\co X \to X\}_{X\in \Ob{\mc B}}$ is a twist for $\mc B$ if and only if $\mc B$ is symmetric.
	Also, any ribbon category (see Section \ref{prelimiribbon}) has a canonical twist.
	By a \textit{braided category with a twist}, we mean a braided category endowed with a twist.

	\subsection{Graphical calculus} \label{cg}
	In this paper, we intensively use the \textit{Penrose graphical calculus}, which allows us to avoid lengthy algebraic computations by using topological arguments.
	The diagrams read from bottom to top.
	In a monoidal category~$\mc B$, the diagrams are made of arcs colored by objects of~$\mc B$ and of boxes, colored by morphisms of~$\mc B$.
	Arcs colored by~$\uu$ may be omitted in the pictures.
	The identity morphism of an object~$X$, a morphism~$f\co X\to Y$ in~$\mc B$, and its composition with a morphism~$g\co Y\to Z$ in~$\mc B$ are represented respectively as
	\[
	\,
	\psfrag{X}[cc][cc][0.75]{$X$}
	\rsdraw{0.45}{1}{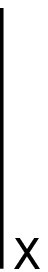}
	\;, \quad
	\,
	\psfrag{X}[cc][cc][0.75]{$X$}
	\psfrag{Y}[cc][cc][0.75]{$Y$}
	\psfrag{f}[cc][cc][0.85]{$f$}
	\rsdraw{0.45}{1}{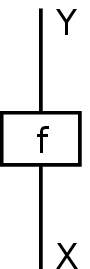}
	\;,\quad \text{and} \quad
	\,
	\psfrag{X}[cc][cc][0.75]{$X$}
	\psfrag{Y}[cc][cc][0.75]{$Y$}
	\psfrag{Z}[cc][cc][0.75]{$Z$}
	\psfrag{f}[cc][cc][0.85]{$f$}
	\psfrag{g}[cc][cc][0.85]{$g$}
	\rsdraw{0.45}{1}{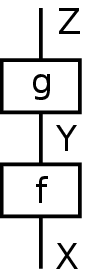}
	\;.\]
	The tensor product of two morphisms~$f\co X\to Y$ and~$g\co U \to V$ is represented by placing a picture of~$f$ to the left of the picture of~$g$:
	\[
	f\tens g
	=
	\,
	\psfrag{X}[cc][cc][0.75]{$X$}
	\psfrag{Y}[cc][cc][0.75]{$Y$}
	\psfrag{U}[cc][cc][0.75]{$U$}
	\psfrag{V}[cc][cc][0.75]{$V$}
	\psfrag{f}[cc][cc][0.85]{$f$}
	\psfrag{g}[cc][cc][0.85]{$g$}
	\rsdraw{0.45}{1}{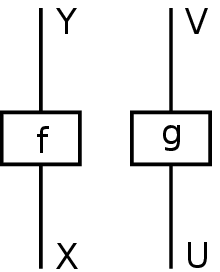}
	\;.
	\]
	Any diagram represents a morphism. The latter depends only on the isotopy class of the diagram representing it. For example, the
	\emph{level-exchange property}
	\[
	\, \psfrag{X}[cc][cc][0.75]{$X$}
	\psfrag{Y}[cc][cc][0.75]{$Y$}
	\psfrag{U}[cc][cc][0.75]{$U$}
	\psfrag{V}[cc][cc][0.75]{$V$}
	\psfrag{f}[cc][cc][0.85]{$f$}
	\psfrag{g}[cc][cc][0.85]{$g$}
	\rsdraw{0.45}{1}{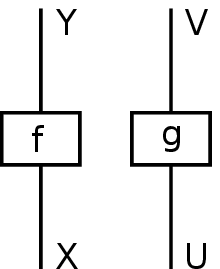}\;= \, \psfrag{X}[cc][cc][0.75]{$X$}
	\psfrag{Y}[cc][cc][0.75]{$Y$}
	\psfrag{U}[cc][cc][0.75]{$U$}
	\psfrag{V}[cc][cc][0.75]{$V$}
	\psfrag{f}[cc][cc][0.85]{$f$}
	\psfrag{g}[cc][cc][0.85]{$g$}
	\rsdraw{0.45}{1}{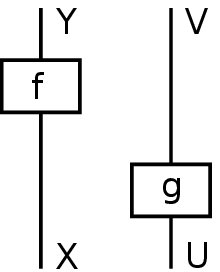}\;=\, \psfrag{X}[cc][cc][0.75]{$X$}
	\psfrag{Y}[cc][cc][0.75]{$Y$}
	\psfrag{U}[cc][cc][0.75]{$U$}
	\psfrag{V}[cc][cc][0.75]{$V$}
	\psfrag{f}[cc][cc][0.85]{$f$}
	\psfrag{g}[cc][cc][0.85]{$g$}
	\rsdraw{0.45}{1}{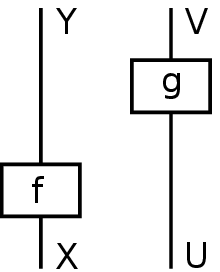}\;. \]
	reflects the formula \[f\tens g =(f\tens \id_V)(\id_X \tens g)=(\id_Y \tens g)(f\tens \id_U).\]
	When $\mc B$ is braided with a braiding $\tau$, we depict
	\[
	\tau_{X,Y}=\,
	\psfrag{X}{$X$}
	\psfrag{Y}{$Y$}
	\rsdraw{0.45}{0.75}{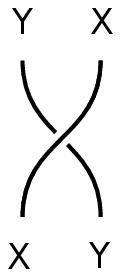}
	\; \quad \text{and} \quad
	\tau_{X,Y}^{-1}=\,
	\psfrag{X}{$Y$}
	\psfrag{Y}{$X$}
	\rsdraw{0.45}{0.75}{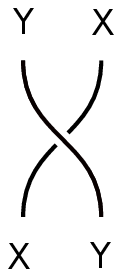}
	\;.
	\]
	When~$\mc B$ is a braided category with a twist~$\theta=\{\theta_X\co X\to X\}_{X\in \Ob{\mc B}}$, we depict
	$$\theta_X = \, \psfrag{X}{$X$}
	\rsdraw{0.45}{0.75}{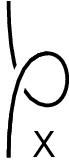}\; \quad \text{and} \quad \left(\theta_X\right)^{-1} = \, \psfrag{X}{$X$}
	\rsdraw{0.45}{0.75}{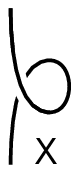}\;.$$
	The defining condition \eqref{twistcondition} for the twist says that for all $X,Y\in \Ob{\mc B}$,
	$$
	\, \psfrag{X}{$X\tens Y$}
	\rsdraw{0.45}{0.75}{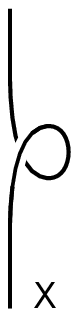}\;\hspace{0.65cm}  = \hspace{0.2cm} \, \psfrag{X}{\hspace{-0.1cm}$X$}
	\psfrag{U}{$Y$}
	\rsdraw{0.45}{0.75}{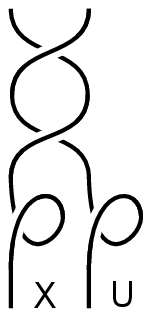}\;.
	$$
	We warn the reader that this notation should not be confused with notation of a left twist in a ribbon category (see Section \ref{prelimiribbon}). We choose this notation since any ribbon category is an important particular example of a braided category with a twist.

	\subsection{Pivotal categories} \label{pivotal}

A \textit{pivotal category} is a monoidal category ~$\mc C$ such that to any object~$X$ of~$\mc C$ is associated a dual object $X^*\in \Ob{\mc C}$ and four morphisms
\begin{align*}
  \evl_X \co X^* \tens X \to \uu, & \quad   \coevl_X \co \uu   \to X\tens X^*, \\
  \evr_X \co X \tens X^* \to \uu, & \quad \coevr_X \co \uu   \to X^* \tens X,
\end{align*}
satisfying several conditions and such that the so called left and right duality functors coincide as monoidal functors. The latter implies in particular that the dual morphism $f^* \co Y^* \to X^*$ of a morphism $f\co X\to Y$ in $\mc C$ is computed by
	\begin{align*}
	f^*&=(\id_{X^*} \tens \evr_Y )(\id_{X^*} \tens f \tens \id_{Y^*})(\coevr_X \tens \id_{Y^*})= \\
	& =(\evl_Y \tens \id_{{X^*}})(\id_{{Y^*}} \tens f \tens \id_{{X^*}})(\id_{{Y^*}} \tens \coevl_X).
	\end{align*}
The graphical calculus for monoidal categories (see Section~\ref{cg}) is extended to pivotal cate\-gories by orienting arcs. If an arc colored by~$X$ is oriented upwards, the represented object in source/target of corresponding morphism is~$X^*$.
	For example,~$\id_{X}, \id_{X^*}$, and a morphism~$f\co X\tens Y^* \tens Z \to U \tens V^*$ are depicted by
	\[
	\id_{X}=
	\,
	\psfrag{X}[cc][cc][0.75]{$X$}
	\rsdraw{0.45}{1}{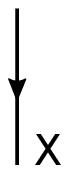}
	\;,\quad
	\id_{X^*}=
	\,
	\psfrag{X}[cc][cc][0.75]{$X$}
	\rsdraw{0.45}{1}{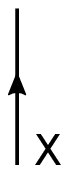}
	\;
	=
	\,
	\psfrag{X}[cc][cc][0.75]{$X^*$}
	\rsdraw{0.45}{1}{idX}
	\; \hspace{0.2cm},
	\quad \text{and} \quad
	f=
	\,
	\psfrag{X}[cc][cc][0.75]{$X$}
	\psfrag{Y}[cc][cc][0.75]{$Y$}
	\psfrag{Z}[cc][cc][0.75]{$Z$}
	\psfrag{U}[cc][cc][0.75]{$U$}
	\psfrag{V}[cc][cc][0.75]{$V$}
	\psfrag{f}[cc][cc][0.85]{\hspace{0.55cm}$f$}
	\rsdraw{0.45}{1}{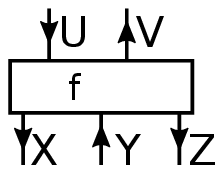}
	\;. \]
	The morphisms $\evl_X, \evr_X, \coevl_X$, and $\coevr_X$ are respectively depicted by
	\[
	\,
	\psfrag{X}[cc][cc][0.75]{$X$}
	\rsdraw{0.45}{1}{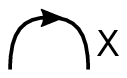}
	\;, \quad
	\,
	\psfrag{X}[cc][cc][0.75]{$X$}
	\rsdraw{0.45}{1}{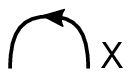}
	\;, \quad
	\,
	\psfrag{X}[cc][cc][0.75]{$X$}
	\rsdraw{0.45}{1}{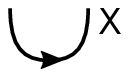}
	\;,  \quad \text{and} \quad
	\,
	\psfrag{X}[cc][cc][0.75]{$X$}
	\rsdraw{0.45}{1}{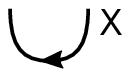}
	\;.
	\]
For more details, see \cite[Chapter 1]{moncatstft}.
	\subsection{Ribbon categories} \label{prelimiribbon}
	Let~$\mc B$ be a braided pivotal category.
	The \emph{left twist} of an object~$X$ of $\mc B$ is defined by
	\[\theta_X^l=  \, \psfrag{X}{$X$}
	\rsdraw{0.45}{0.75}{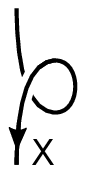}\;= (\id_X\tens \evr_X)(\tau_{X,X}\tens \id_{X^*})(\id_X\tens \coevl_X)\co X\to X,\]
	while the \emph{right twist} of $X$ is defined by
	\[\theta_X^r=\, \psfrag{X}{$X$}
	\rsdraw{0.45}{0.75}{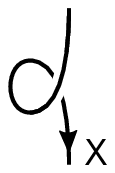}\;=(\evl_X\tens \id_{X})(\id_{X^*}\tens \tau_{X,X})(\coevr_X\tens \id_X)\co X\to X.\]
	The left and the right twist are natural isomorphisms with inverses
	$$
	(\theta_X^l)^{-1}=\, \psfrag{X}{$X$}
	\rsdraw{0.45}{0.75}{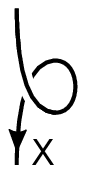}\; \quad \text{and} \quad
	(\theta_X^r)^{-1}=\, \psfrag{X}{$X$}
	\rsdraw{0.45}{0.75}{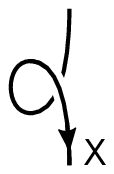}\;.
	$$
	A \textit{ribbon category} is a braided pivotal category $\mc B$ such that $\theta_X^l=\theta_X^r$  for all $X \in \Ob{\mc B}$.
	In this case, the family~$\theta=\{\theta_X=\theta_X^l=\theta_X^r \co X\to X\}_{X\in \Ob{\mc{B}}}$ is a twist in the sense of Section~\ref{catswithtwist} and is called
	the {\it twist} of $\mc B$.

	\section{Algebraic cyclic theories} \label{algebraiccyclic}
	In this section, $\mc B$ denotes a braided category with a twist $\theta$. We review some constructions of (co)cyclic sets from (co)algebras in $\mc B$.

	\subsection{A cocyclic set from coalgebras}
	\label{sec-coalg}
	A~\textit{coalgebra} in~$\mc B$ is a triple $(C, \Delta, \varepsilon)$, where~$C$ is an object of~$\mc B$,~$\Delta \co C\to C\tens C$ and~$\varepsilon\co C \to \uu$ are morphisms in~$\mc B$, called~\textit{comultiplication} and~\textit{counit} respectively, which satisfy
	\begin{equation*}
	(\Delta\tens \id_C)\Delta= (\id_C \tens \Delta)\Delta \quad \text{and} \quad  (\id_C \tens \varepsilon)\Delta=\id_C= (\varepsilon \tens \id_C)\Delta.
	\end{equation*}
	The comultiplication and the counit are depicted by
	$$\,
	\psfrag{C}[cc][cc][0.75]{$C$}
	\rsdraw{0.45}{1}{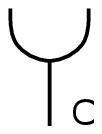}
	\;
	\quad \text{and} \quad \,
	\psfrag{C}[cc][cc][0.75]{$C$}
	\rsdraw{0.45}{1}{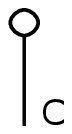}
	\;.$$
	
	Any coalgebra~$C$ in~$\mc B$ gives rise to a cocyclic set $C^\bullet$ as follows. For any~$n \in \N$, define~$C^n=\Hom_\mc{B}(C^{\tens n+1}, \uu).$
	Next, define the cofaces $\{\delta_i^n \co C^{n-1} \to C^n\}_{n\in \N^*, 0\le i \le n}$, the codegeneracies~$\{\sigma_j^n \co C^{n+1} \to C^n\}_{n\in \N, 0\le j \le n}$, and the cocyclic operators~$\{\tau_n \co C^n \to C^n\}_{n\in \N}$ by set\-ting
	\begin{gather*}
	\delta_i^n (f) =
	\,
	\psfrag{b}[cc][cc][0.75]{\hspace{-0.2cm}$0$}
	\psfrag{i}[cc][cc][0.75]{$i$}
	\psfrag{e}[cc][cc][0.75]{$n$}
	\psfrag{T}[cc][cc][0.85]{\hspace{1cm}$f$}
	\rsdraw{0.55}{1}{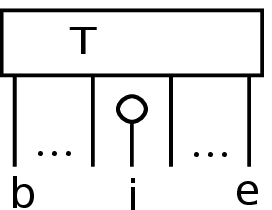}\;, \quad \sigma_j^n (f) =
	\,
	\psfrag{b}[cc][cc][0.75]{\hspace{-0.2cm}$0$}
	\psfrag{j}[cc][cc][0.75]{$j$}
	\psfrag{e}[cc][cc][0.75]{$n$}
	\psfrag{T}[cc][cc][0.85]{\hspace{1.5cm}$f$}
	\rsdraw{0.55}{1}{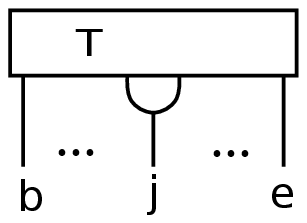}\;,\quad \tau_n(f)=\,
	\psfrag{b}[cc][cc][0.75]{\hspace{-0.1cm}$0$}
	\psfrag{bn}[cc][cc][0.75]{\hspace{-0.4cm}$1$}
	\psfrag{er}[cc][cc][0.75]{\hspace{0.4cm}$n-1$}
	\psfrag{e}[cc][cc][0.75]{$n$}
	\psfrag{T}[cc][cc][0.85]{\hspace{0.1cm}$f$}
	\rsdraw{0.55}{1}{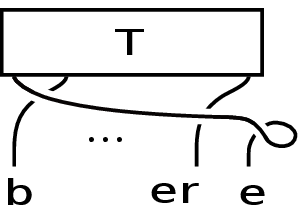}\;.
	\end{gather*}
	Here~$\tau_0=\id_{C^0}$. An integer~$k$ below an arc denotes the~$k$-th tensorand of a tensor power of~$C$.
	\begin{lem} \label{coalg}
	The family $C^\bullet=\{C^n\}_{n\in \N}$ endowed with the cofaces $\{\delta_i^n\}_{n\in \N^*, 0 \le i \le n}$, the codegeneracies $\{\sigma_j^n\}_{n\in \N, 0 \le j \le n}$, and the cocyclic operators $\{\tau_n\}_{n\in \N}$ is a cocyclic set.
	\end{lem}
	Lemma~\ref{coalg} is proved in Section~\ref{Proofscoalg}.
	
	\subsection{A cyclic set from algebras}
	\label{sec-alg}
	An~\textit{algebra} in~$\mc B$ is a triple~$(A,m,u)$, where~$A$ is an object of~$\mc B$,~$m\co A\tens A \to A$ and~$u \co \uu \to A$ are morphisms in~$\mc B$, called~\textit{multiplication} and~\textit{unit} respectively, which satisfy
	\begin{equation*}
	m(m\tens \id_A)=m(\id_A \tens m) \quad \text{and} \quad m(u \tens \id_A)= \id_A= m(\id_A\tens u).
	\end{equation*}
	The multiplication and the unit are depicted by
	$$\,
	\psfrag{A}[cc][cc][0.75]{$A$}
	\rsdraw{0.45}{1}{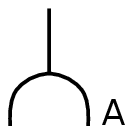}
	\;
	\quad \text{and} \quad \,
	\psfrag{A}[cc][cc][0.75]{$A$}
	\rsdraw{0.45}{1}{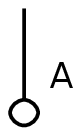}
	\;.$$
	
	Any algebra~$A$ in~$\mc B$ gives rise to a cyclic set~$A_\bullet$ as follows. For any~$n \in \N$, define~$A_n=\Hom_\mc{B}(A^{\tens n+1}, \uu)$.
	Next, define the faces~$\{d_i^n\co A_n \to A_{n-1} \}_{n\in \N^*, 0\le i \le n}$, the dege\-neracies~$\{s^n_j\co A_n \to A_{n+1}\}_{n \in \N, 0\le j \le n}$, and cyclic operators~$\{t_n\co A_n \to A_n\}_{n\in \N}$ by setting
	\begin{gather*}
	d_0^n(f)=
	\, \psfrag{b}[cc][cc][0.75]{$0$}
	\psfrag{e}[cc][cc][0.75]{$n-1$}
	\psfrag{T}[cc][cc][0.85]{\hspace{-0.2cm}$f$}
	\rsdraw{0.55}{1}{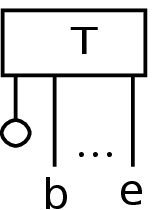} \;, \quad d_i^n(f)=\, \psfrag{b}[cc][cc][0.75]{$0$}
	\psfrag{i}[cc][cc][0.75]{$i-1$}
	\psfrag{k}[cc][cc][0.75]{\hspace{-0.2cm}$i$}
	\psfrag{e}[cc][cc][0.75]{$n-1$}
	\psfrag{T}[cc][cc][0.85]{\hspace{-0.08cm}$f$}
	\rsdraw{0.55}{1}{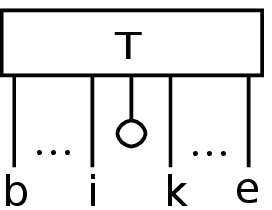} \;, \quad
	d_n^n(f)=\, \psfrag{b}[cc][cc][0.75]{$0$}
	\psfrag{e}[cc][cc][0.75]{$n-1$}
	\psfrag{T}[cc][cc][0.85]{\hspace{-0.2cm}$f$}
	\rsdraw{0.55}{1}{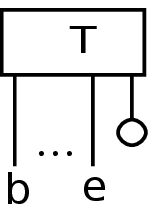} \;, \\
	s_j^n(f)=\,
	\psfrag{b}[cc][cc][0.75]{\hspace{-0.2cm}$0$}
	\psfrag{j}[cc][cc][0.75]{$j$}
	\psfrag{l}[cc][cc][0.75]{$j+1$}
	\psfrag{e}[cc][cc][0.75]{$n+1$}
	\psfrag{T}[cc][cc][0.85]{$f$}
	\rsdraw{0.55}{1}{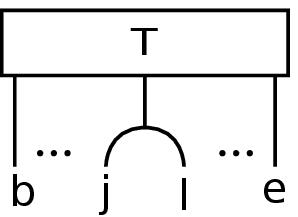}
	\;, \quad t_n(f)=\, \psfrag{b}[cc][cc][0.75]{$0$}
	\psfrag{bn}[cc][cc][0.75]{\hspace{-0.4 cm}$1$}
	\psfrag{er}[cc][cc][0.75]{$1$}
	\psfrag{e}[cc][cc][0.75]{$n$}
	\psfrag{T}[cc][cc][0.85]{$f$}
	\rsdraw{0.55}{1}{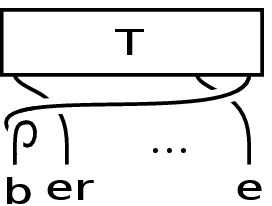} \;.
	\end{gather*}
	Here $t_0=\id_{A_0}$.
	
	\begin{lem} \label{alg}
	The family $A_\bullet = \{A_n\}_{n\in \N}$ endowed with the faces $\{d_i^n\}_{n\in \N^*, 0 \le i \le n}$, the degene\-racies $\{s_j^n\}_{n\in \N, 0 \le j \le n}$, and the cyclic operators $\{t_n\}_{n\in \N}$ is a cyclic set.
	\end{lem}
	The proof of Lemma \ref{alg} is similar to the proof of Lemma \ref{coalg}.
	\subsection{Cyclic duals} \label{passing}
	The cyclic duality $L\co \Delta C^\opp \to \Delta C$ from Section \ref{connes loday dual} transforms the cocyclic set $C^\bullet$ from Lemma \ref{coalg} into the cyclic set $C^\bullet\circ L$. For any $n\in \N$, $C^\bullet\circ L(n)=C^n=\Hom_{\mc B}(C^{\tens n+1}, \uu)$.
The faces $\{\tilde{d}_i^n \co \Hom_{\mc B}(C^{\tens n+1}, \uu) \to \Hom_{\mc B}(C^{\tens n}, \uu)\}_{n\in \N^*, 0\le i \le n}$, the degeneracies $\{\tilde{s}_j^n \co \Hom_{\mc B}(C^{\tens n+1}, \uu)\to \Hom_{\mc B}(C^{\tens n+2}, \uu)\}_{n \in \N, 0\le j \le n}$, and the cyclic operators $\{\tilde{t}_n \co \Hom_{\mc B}(C^{\tens n+1}, \uu) \to\Hom_{\mc B}(C^{\tens n+1}, \uu)\}_{n \in \N}$ are computed by setting
	\begin{gather*}
	\tilde{d}_i^n (f) =
	\,
	\psfrag{b}[cc][cc][0.75]{\hspace{-0.1cm}$1$}
	\psfrag{j}[cc][cc][0.75]{$i+1$}
	\psfrag{e}[cc][cc][0.75]{$n$}
	\psfrag{T}[cc][cc][0.85]{\hspace{1.3cm}$f$}
	\rsdraw{0.55}{1}{HDsj}\;, \quad
	\tilde{d}_n^n(f)=\,
	\psfrag{b}[cc][cc][0.75]{\hspace{-0.1cm}$1$}
	\psfrag{bn}[cc][cc][0.75]{$2$}
	\psfrag{e}[cc][cc][0.75]{\hspace{0.1cm}$n$}
	\psfrag{T}[cc][cc][0.85]{\hspace{0.3cm}$f$}
	\rsdraw{0.55}{1}{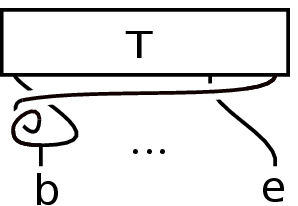}\;, \\
	\tilde{s}_j^n(f)=
	\psfrag{b}[cc][cc][0.75]{\hspace{-0.1cm}$0$}
	\psfrag{i}[cc][cc][0.75]{$j+1$}
	\psfrag{e}[cc][cc][0.75]{$n+1$}
	\psfrag{T}[cc][cc][0.85]{\hspace{0.9cm}$f$}
	\rsdraw{0.55}{1}{Edeltai}\;, \quad
	\tilde{t}_n(f)=\, \psfrag{b}[cc][cc][0.75]{$0$}
	\psfrag{bn}[cc][cc][0.75]{\hspace{-0.3cm}$1$}
	\psfrag{er}[cc][cc][0.75]{$1$}
	\psfrag{e}[cc][cc][0.75]{$n$}
	\psfrag{T}[cc][cc][0.85]{$f$}
	\rsdraw{0.55}{1}{Etn} \;.
	\end{gather*}
	
	Similarly as above, the functor~$L^{\opp}\co \Delta C\to \Delta C^{\opp}$ transforms the cyclic set~$A_\bullet$ from Lemma~\ref{alg} into the cocyclic set~$A_\bullet \circ L^{\opp}$. By definitions,~$A_\bullet \circ L^\opp(n)=A_n=\Hom_{\mc B}(A^{\tens n+1}, \uu)$ for all~$n\in \N$.
The cofaces~$\{\tilde{\delta}_i^n\co\Hom_{\mc B}(A^{\tens n}, \uu) \to \Hom_{\mc B}(A^{\tens n+1}, \uu)\}_{n\in \N^*, 0\le i \le n}$, the codegeneracies~$\{\tilde{\sigma}_j^n \co \Hom_{\mc B}(A^{\tens n+2}, \uu) \to\Hom_{\mc B}(A^{\tens n+1}, \uu)\}_{n\in \N, 0\le j \le n}$, and the cocyclic operators~$\{\tilde{\tau}_n \co  \Hom_{\mc B}(A^{\tens n+1}, \uu) \to \Hom_{\mc B}(A^{\tens n+1}, \uu)\}_{n\in \N}$ are computed by
	\begin{gather*}
	\tilde{\delta}_i^n(f) = \, \psfrag{b}[cc][cc][0.75]{\hspace{-0.1cm}$0$}
	\psfrag{j}[cc][cc][0.75]{$i$}
	\psfrag{l}[cc][cc][0.75]{$i+1$}
	\psfrag{e}[cc][cc][0.75]{$n$}
	\psfrag{T}[cc][cc][0.85]{$f$}
	\rsdraw{0.55}{1}{Esj} \;, \quad
	\tilde{\delta}_n^n(f) =\, \psfrag{b}[cc][cc][0.75]{\hspace{-0.1cm}$0$}
	\psfrag{bn}[cc][cc][0.75]{\hspace{-0.1cm}$1$}
	\psfrag{be}[cc][cc][0.75]{$n-1$}
	\psfrag{e}[cc][cc][0.75]{$n$}
	\psfrag{T}[cc][cc][0.85]{\hspace{1.25cm}$f$}
	\rsdraw{0.55}{1}{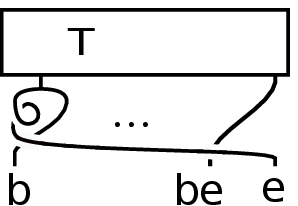} \;, \quad \\ \tilde{\sigma}_j^n(f)=\,
	\psfrag{b}[cc][cc][0.75]{$0$}
	\psfrag{i}[cc][cc][0.75]{\hspace{-0.1 cm}$j$}
	\psfrag{k}[cc][cc][0.75]{\hspace{-0.4 cm}$j+1$}
	\psfrag{e}[cc][cc][0.75]{$n$}
	\psfrag{T}[cc][cc][0.85]{$f$}
	\rsdraw{0.55}{1}{Edi}\;, \quad
	\tilde{\tau}_n(f)=\, \psfrag{b}[cc][cc][0.75]{\hspace{-0.1cm}$0$}
	\psfrag{bn}[cc][cc][0.75]{\hspace{-0.4 cm}$1$}
	\psfrag{er}[cc][cc][0.75]{\hspace{0.2 cm}$n-1$}
	\psfrag{e}[cc][cc][0.75]{$n$}
	\psfrag{T}[cc][cc][0.85]{$f$}
	\rsdraw{0.55}{1}{HDtn} \;.
	\end{gather*}
	Note that the construction~$A_\bullet \circ L^\opp$ is a particular case of the work of Akrami and Majid~\cite{cycliccocycles} (since any algebra in a braided category with a twist is a ribbon algebra in the sense of~\cite{cycliccocycles}).
	
	\subsection{Proof of Lemma \ref{coalg}}\label{Proofscoalg}
	The cofaces~$\{\delta_{i}^n\}_{n\in \N^*, 0\le i \le n}$, the codegeneracies $\{\sigma_j^n\}_{n\in\N, 0\le j \le n}$, and the cocyclic operators~$\{\tau_n\}_{n\in \N}$ of~$C^\bullet$ are given by formulas~$\delta_i^n(f)=fd_i^n$,~$\sigma_j^n(f)=fs_j^n$, and~$\tau_n(f)=ft_n$, where
	\begin{gather*}
	d_i^n=\psfrag{b}{$0$}
	\psfrag{i}{\hspace{-0.05cm}$i$}
	\psfrag{e}{$n$}
	\rsdraw{0.55}{0.75}{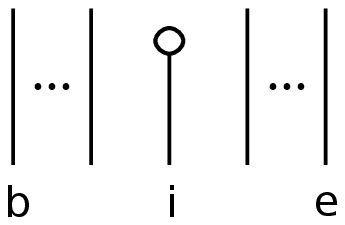} \;, \quad s_j^n=\,
	\psfrag{b}{$0$}
	\psfrag{j}{$j$}
	\psfrag{e}{$n$}
	\rsdraw{0.55}{0.75}{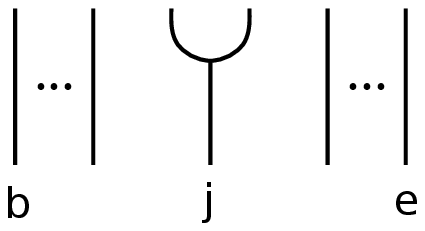}\;,  \quad t_0= \theta_C^{-1}, \quad t_n = \, \psfrag{b}{$0$}
	\psfrag{bn}{\hspace{0.1cm}$1$}
	\psfrag{pe}[l]{\hspace{-0.2cm}$n-1$}
	\psfrag{e}{\hspace{-0.1cm}$n$}
	\rsdraw{0.55}{0.75}{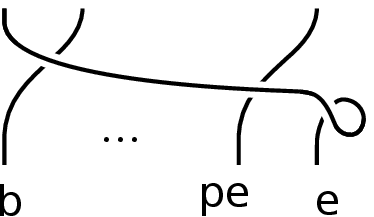}  \;.
	\end{gather*}
	We claim that morphisms~$\{d_i^n\}_{n\in \N^*, 0\le i \le n}$,~$\{s_j^n\}_{n\in \N, 0\le j \le n}$, and~$\{t_n\}_{n\in \N}$ satisfy \eqref{faces}-\eqref{s0tn} and for all~$n \in \N$, the following ``twisted cyclicity condition'':
	\begin{equation} \label{twistedcyc}
	t_n^{n+1}=(\theta_{C^{\tens n+1}})^{-1}.
	\end{equation}
	This indeed implies that~$C^\bullet$ is a cocyclic set. Here, the cocyclicity condition~\eqref{cocyclicity} for~$\tau_n$ follows by the condition~\eqref{twistedcyc}, naturality of the twist of~$\mc B$, and the fact that~$\theta_{\uu}=\id_\uu.$ Indeed, for any~$n \in \N$ and any~$f \in \Hom_{\mc B}(C^{\tens n+1}, \uu)$,
	$$\tau_n^{n+1}(f)= ft_n^{n+1}=f(\theta_{C^{\tens n+1}})^{-1}=(\theta_{\uu})^{-1}f=f.$$
	
	Let us now check the fact that morphisms $\{d_i^n\}_{n\in \N^*, 0\le i \le n}$, $\{s_j^n\}_{n\in \N, 0\le j \le n}$, and $\{t_n\}_{n\in \N}$ satisfy~\eqref{faces}-\eqref{s0tn} and \eqref{twistedcyc}. Let~$n\ge 1$ and~$0\le i < j \le n+1$. By the level-exchange property, we have
	\begin{equation*}
	d_i^nd_j^{n+1}=
	\,
	\psfrag{b}{$0$}
	\psfrag{i}{$i$}
	\psfrag{j}{$j$}
	\psfrag{e}{\hspace{-0.37cm}$n+1$}
	\rsdraw{0.55}{0.75}{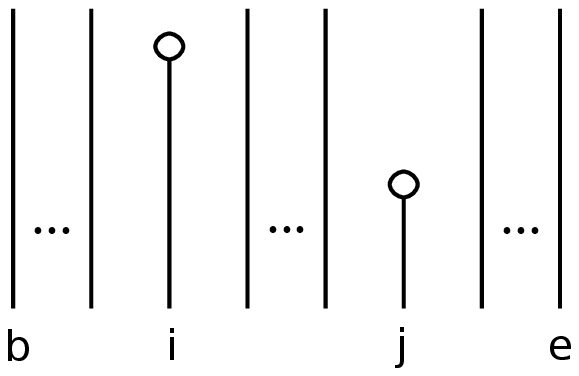}
	\; \quad = \quad
	\,
	\psfrag{b}{$0$}
	\psfrag{i}{$i$}
	\psfrag{j}{$j$}
	\psfrag{e}{\hspace{-0.37cm}$n+1$}
	\rsdraw{0.55}{0.75}{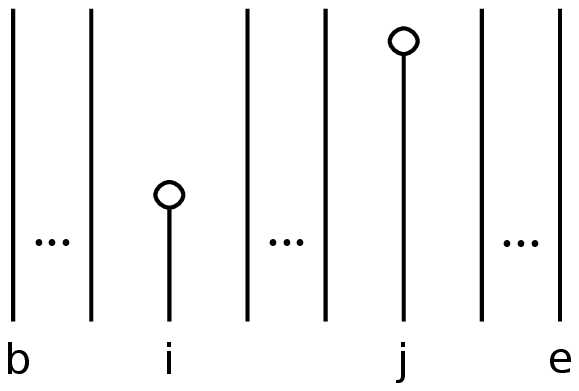}\;
	=d_{j-1}^nd_i^{n+1},
	\end{equation*}
	whence the relation \eqref{faces}.
	
	We now check the relation \eqref{degeneracies}.
	Let $n\ge 0$ and $i\le j$. If $i<j$, we have by the level-exchange property that	
	\begin{equation*}s_i^{n+1}s_j^n=\,
	\psfrag{b}{$0$}
	\psfrag{i}{$i$}
	\psfrag{j}{$j$}
	\psfrag{e}{$n$}
	\rsdraw{0.55}{0.75}{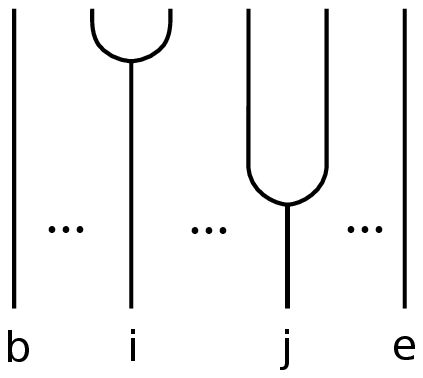}\;\quad = \quad
	\,
	\psfrag{b}{$0$}
	\psfrag{i}{$i$}
	\psfrag{j}{$j$}
	\psfrag{e}{$n$}
	\rsdraw{0.55}{0.75}{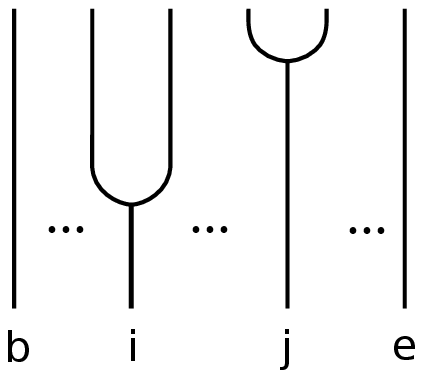}\; =s_{j+1}^{n+1}s_i^n.
	\end{equation*} In the case when $i=j$, we have by the coassociativity of $C$ that
	\begin{equation*}s_i^{n+1}s_i^n=\,
	\psfrag{b}{$0$}
	\psfrag{i}{$i$}
	\psfrag{j}{$j$}
	\psfrag{e}{$n$}
	\rsdraw{0.55}{0.75}{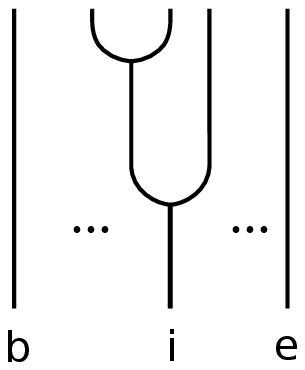}\;\quad = \quad
	\,
	\psfrag{b}{$0$}
	\psfrag{i}{\hspace{-0.1cm}$i$}
	\psfrag{j}{$j$}
	\psfrag{e}{$n$}
	\rsdraw{0.55}{0.75}{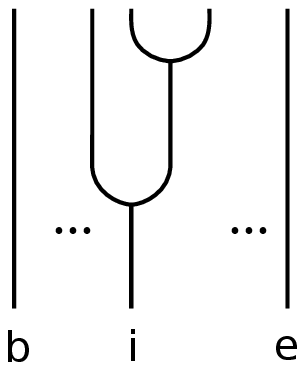}\; =s_{i+1}^{n+1}s_i^n.
	\end{equation*}
	
	Let us now show that relations \eqref{compfacdeg} hold. If $i<j$ or $i>j+1$, this follows by the level-exchange property.
	Further, if $i=j$, we have by the counitality of $C$ that
	\[
	d_i^{n+1}s_i^n=
	\,
	\psfrag{b}{$0$}
	\psfrag{i}{$i$}
	\psfrag{e}{$n$}
	\rsdraw{0.55}{0.75}{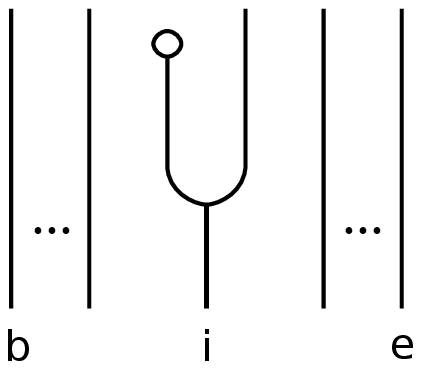}=\id_{C^{\tens n+1}}.
	\;
	\]
	The case when $i=j+1$ also follows by the counitality of $C$.
	
	Let us now check the relation \eqref{compcycfac}.
	Let $n\ge 1$ and $1\le i \le n$.
	By the naturality of the braiding in $\mc B$, we have
	\[
	d_i^nt_n=
	\,
	\psfrag{b}{$0$}
	\psfrag{bn}{$1$}
	\psfrag{i}{\hspace{-0.1cm}$i-1$}
	\psfrag{pe}{$n-1$}
	\psfrag{e}{$n$}
	\rsdraw{0.55}{0.75}{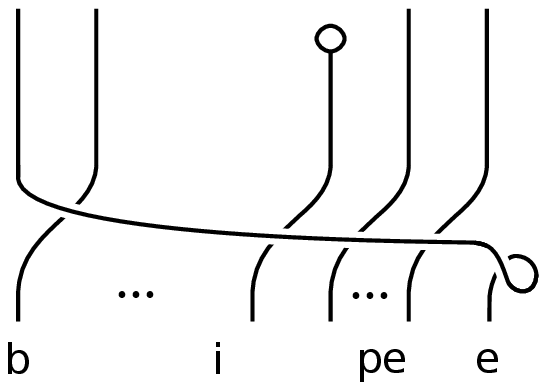}
	\;\quad= \quad  \,
	\psfrag{b}{$0$}
	\psfrag{bn}{$1$}
	\psfrag{i}{$i-1$}
	\psfrag{pe}{$n-1$}
	\psfrag{e}{$n$}
	\rsdraw{0.55}{0.75}{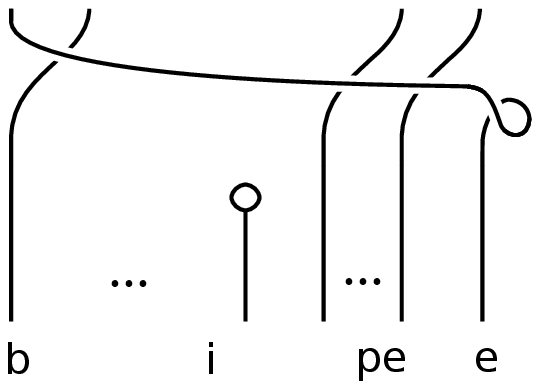}
	\; =t_{n-1}d_{i-1}^n,
	\]
	whence the relation \eqref{compcycfac}.
	The relation \eqref{d0tn} follows by the naturality of the braiding, naturality of the twist morphism, and the fact that $\theta_\uu=\id_\uu$.
	The relation \eqref{compcycdeg} follows by the naturality of the braiding and the level-exchange property.
	
	By the naturality of the braiding and the equation \eqref{twistcondition}, we have
	\begin{align*}
	s_0^nt_n&=\,
	\psfrag{b}{$0$}
	\psfrag{bn}{\hspace{0.1cm}$1$}
	\psfrag{pe}{$n-1$}
	\psfrag{e}{$n$}
	\rsdraw{0.55}{0.75}{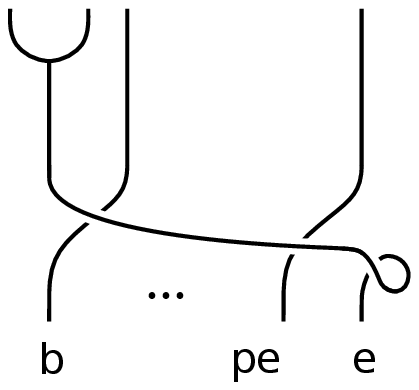}\;\quad=\quad
	\,
	\psfrag{b}{$0$}
	\psfrag{bn}{\hspace{0.1cm}$1$}
	\psfrag{pe}{$n-1$}
	\psfrag{e}{$n$}
	\rsdraw{0.55}{0.75}{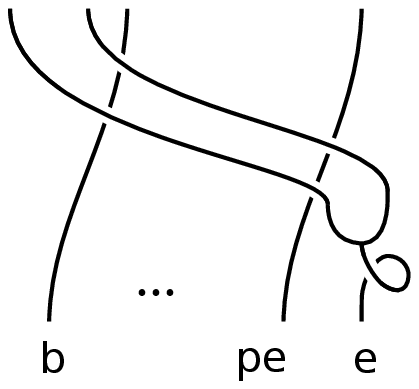}\;\\& \\
	&=\,
	\psfrag{b}{$0$}
	\psfrag{bn}{$1$}
	\psfrag{pe}{$n-1$}
	\psfrag{e}{$n$}
	\rsdraw{0.55}{0.75}{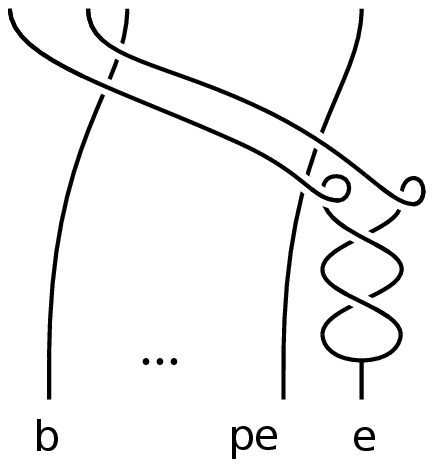}\;\quad=\quad
	\,
	\psfrag{b}{$0$}
	\psfrag{bn}{$1$}
	\psfrag{pe}{$n-1$}
	\psfrag{e}{$n$}
	\rsdraw{0.55}{0.75}{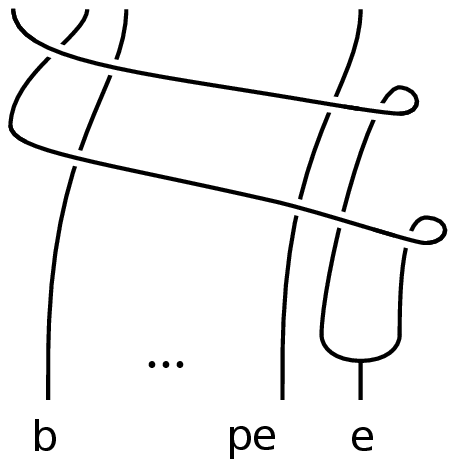}\;\quad=\quad t_{n+1}^2s_n^n,
	\end{align*}
	whence the relation~\eqref{s0tn} in the case when~$n\ge 1$. For~$n=0$, this follows by the equation~\eqref{twistcondition}.
	
	Finally, let us verify the ``twisted cyclicity condition'' \eqref{twistedcyc}.
	In the case~$n=1$ (the general case is treated similarly), this follows by naturality of the braiding and by equation~\eqref{twistcondition}:
	\[t_1^2=\quad\,
	\psfrag{b}{$0$}
	\psfrag{bn}{$0$}
	\psfrag{e}{$1$}
	\rsdraw{0.45}{0.75}{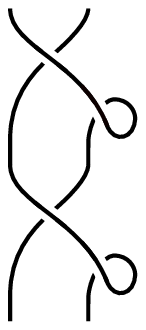}\;\quad=\quad
	\,
	\psfrag{b}{$0$}
	\psfrag{bn}{$0$}
	\psfrag{e}{$1$}
	\rsdraw{0.45}{0.75}{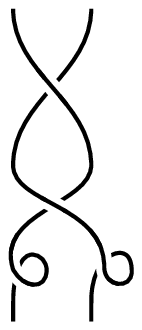}\;=\left(\theta_{C^{\tens 2}}\right)^{-1}.\]
	\qed

	\section{Relation with quantum invariants} \label{relquantum}
	
	In this section we relate the (co)cyclic sets constructed via ribbon string links (in Section~\ref{slrelatedobjects}) to (co)cyclic sets (as in Section~\ref{algebraiccyclic}) associated to the coend of a ribbon category (which is a Hopf algebra object). The relationship is given by the quantum invariants à la Reshetikhin-Turaev.
	
	\subsection{Ribbon handles} \label{ribbon handles}
	Recall the notion of a ribbon from Section \ref{sldef}. Let~$n$ be a non-negative integer.
	A \textit{ribbon}~$n$-\emph{handle} is an oriented surface~$H$ embedded in the strip~$\R^2 \times \left[0,1\right]$ and decomposed into a dis\-joint u\-ni\-on of~$n$ ribbons such that~$H\cap \R^2\times \{1\} = \emptyset$ and so that~$H$ meets~$\R^2\times \{0\}$ orthogonally as follows. For all~$1\le k \le n$, the bottom base of the~$k$-th ribbon of~$H$ is the segment~$\left[2k-1-\frac{1}{4},2k-1+\frac{1}{4}\right]\times \{0\} \times \{0\}$, the top base of the~$k$-th ribbon of~$H$ is the segment~$\left[2k-\frac{1}{4},2k+\frac{1}{4}\right]\times \{0\} \times \{0\}$, and in the points of these segments, the orientation of~$T$ is determined by the vector~$(1,0,0)$ tangent to~$H$.
By an \textit{isotopy of ribbon handles}, we mean isotopy in $\R^2 \times [0,1]$ constant on the boundary and preserving splitting into ribbons as well as the orientation of the surface. As in Section \ref{sldef}, we present a ribbon~$n$-handle by a planar diagram with a blackboard framing convention:
	$$
	\,
	\psfrag{b}[cc][cc][0.75]{$1$}
	\psfrag{e}[cc][cc][0.75]{\hspace{0.1cm}$n$}
	\psfrag{T}[cc][cc][0.85]{$H$}
	\rsdraw{0.55}{1}{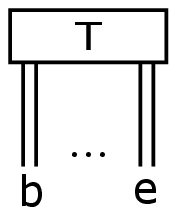}\;.$$
	As shown, we number its ribbons from left to the right.
	
	The category~$\textbf{RH}$ of ribbon handles has as objects non-negative integers.
	For two non-negative integers $m$ and $n$, the set of morphisms from $m$ to $n$ is defined by
	\[\Hom_{\textbf{RH}}(m,n)= \begin{cases}
	\text{isotopy classes of ribbon $n$-handles} \quad &\text{if } m=n, \\
	\emptyset \quad& \text{if } m\neq n.
	\end{cases}\]
	The composition~$H\circ H'$ of two ribbon~$n$-handles $H$ and~$H'$ and the identity for this composition are defined by
	$$
	H
	\circ
	H'=
	\,\psfrag{b}[cc][cc][0.75]{$1$}
	\psfrag{e}[cc][cc][0.75]{\hspace{0.1cm}$n$}
	\psfrag{T}[cc][cc][0.85]{$H$}
	\psfrag{H}[cc][cc][0.85]{$H'$}
	\rsdraw{0.55}{1}{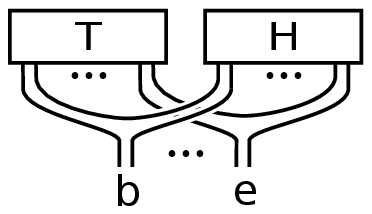}\; \quad \text{and} \quad \id_n
	= \,\psfrag{b}[cc][cc][0.75]{$1$}
	\psfrag{e}[cc][cc][0.75]{$n$}
	\rsdraw{0.55}{1}{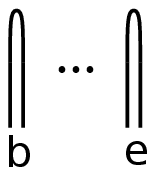}\;.
	$$
	
	Let us recall the construction of the mutually inverse functors~$\mathfrak{F}\co \textbf{RSL} \to \textbf{RH}$ and $\mathfrak{G}\co \textbf{RH} \to \textbf{RSL}$ from \cite{hopfdiag}.
	For any non-negative integer~$n$, set~$\mathfrak{F}(n)=n$ and~$\mathfrak{G}(n)=n$.
	For an isotopy class of a ribbon~$n$-string link~$T$ and an isotopy class of a ribbon~$n$-handle~$H$, set
	\[\mathfrak{F}(T)=
	\,
	\psfrag{b}[cc][cc][0.75]{$1$}
	\psfrag{n}[cc][cc][0.75]{$2$}
	\psfrag{e}[cc][cc][0.75]{$n$}
	\psfrag{T}[cc][cc][0.85]{$T$}
	\rsdraw{0.55}{1}{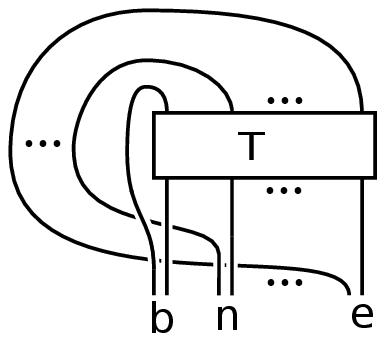}\; \quad \text{and} \quad
	\mathfrak{G}(H)=\,
	\psfrag{b}[cc][cc][0.75]{\hspace{0.35cm}$1$}
	\psfrag{n}[cc][cc][0.75]{$2$}
	\psfrag{e}[cc][cc][0.75]{\hspace{-0.1cm}$n$}
	\psfrag{T}[cc][cc][0.85]{$H$}
	\rsdraw{0.55}{1}{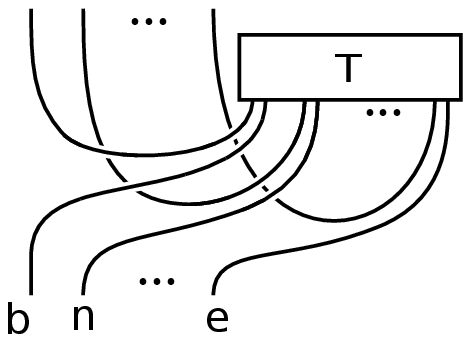}\;.\]
	
	\subsection{Con\-vo\-lu\-tion ca\-te\-go\-ry} \label{conv}
	Let $A=(A,\mu,\eta)$ be an algebra and $C=(C,\Delta,\varepsilon)$ a coalgebra in a braided category~$\mc B$ (see Sections \ref{sec-coalg} and \ref{sec-alg}). The \textit{convolution category}~$\text{Conv}_{\mc B}(C,A)$ is defined as fol\-lows. Its ob\-jects are the non-ne\-ga\-tive in\-te\-gers.
	For two non-negative integers~$m$ and~$n$, the set of morphism from~$m$ to~$n$ is defined by
	\[
	\Hom_{\text{Conv}_{\mc B}(C,A)}(m,n)=
	\begin{cases}
	\Hom_{\mc B}(C^{\tens n}, A) \quad &\text{if } m=n, \\
	\emptyset \quad& \text{if } m\neq n.
	\end{cases}
	\]
	The composition of morphisms is given by the convolution product~$*$, which is defined as follows. For two morphisms~$f,g \in \Hom_{\mc B}(C^{\tens n}, A)$, we set
	$$f*g=\mu(f\tens g)\Delta_{C^{\tens n}},$$
	where~$\Delta_{C^{\tens n}}$ denotes the coproduct on~$C^{\tens n}$ (see \cite[Exercise 6.1.7]{moncatstft}).
	The identity of an object~$n \in\N$ is given by~$\id_n=\eta\varepsilon^{\tens n}$.
	
	\subsection{Coend of a category} \label{coendprelimini} Let~$\mc B$ be a pivotal category. Let~$F_{\mc B} \co \mc B^{\opp} \times \mc B \to \mc B$ be the functor defined by~$F_{\mc B}(X,Y)=X^*\tens Y$.
	A \textit{dinatural transformation} from~$F_{\mc B}$ to an object~$D$ of~$\mc B$ is a function~$d$ that assigns to any object~$X$ of~$\mc B$ a morphism~$d_X\co X^*\tens X \to D$ such that for all morphisms~$f\co X \to Y$ in $\mc B$,
	$$d_X(f^* \tens \id_{X})=d_Y(\id_{Y^*}\tens f).$$
	\textit{The coend of} $\mc B$, if it exists, is a pair~$(\coend,i)$ where~$\coend$ is an object of~$\mc B$ and~$i$ is a dinatural transformation from~$F_{\mc B}$ to~$\coend$, which is universal among all dinatural transformations.
	More precisely, for any dinatural transformation~$d$ from~$F_{\mc B}$ to~$D$, there exists a unique morphism~$\varphi\co \coend \to D$ in~$\mc B$ such that~$d_X=\varphi i_X$ for all~$X \in \Ob{\mc B}$.
	A coend~$(\coend,i)$ of a category~$\mc B$, if it exists, is unique up to a unique isomorphism commuting with the dinatural transformation.
	
	We depict the dinatural transformation $i=\{i_X\co X^*\tens X \to \coend\}_{X\in \Ob{\mc B}}$ as
	\[
	i_X=\,
	\psfrag{X}[cc][cc][0.75]{$X$}
	\psfrag{G}[cc][cc][0.75]{$X$}
	\psfrag{C}[cc][cc][0.75]{\hspace{0.3cm}$\coend$}
	\rsdraw{0.45}{1}{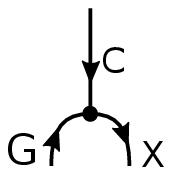}
	\;.
	\]
	The coend of~$\mc B$, if it exists, is a coalgebra in~$\mc B$ with comultiplication~$\Delta \co \coend \to  \coend \tens \coend$ and counit~$\varepsilon \co \coend \to \uu$, which are unique morphisms such that, for all~$X \in \Ob{\mc B}$,
	$$
	\,
		\psfrag{G}[cc][cc][0.75]{$X$}
	\psfrag{X}[cc][cc][0.75]{$X$}
	\psfrag{Y}[cc][cc][0.75]{$Y$}
	\psfrag{C}[cc][cc][0.75]{$\coend$}
	\psfrag{f}[cc][cc][0.75]{$\Delta$}
	\rsdraw{0.45}{1}{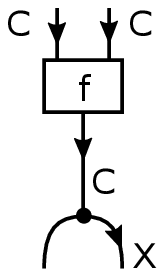}
	\;=~~
	\,
	\psfrag{X}[cc][cc][0.75]{$X$}
	\psfrag{Y}[cc][cc][0.75]{$Y$}
	\psfrag{C}[cc][cc][0.75]{$\coend$}
	\rsdraw{0.45}{1}{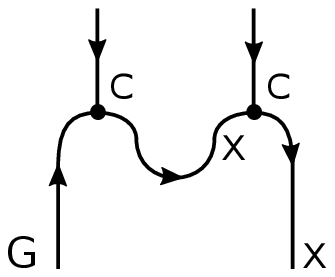}
	\; \quad \text{and} \quad
	\,
	\psfrag{X}[cc][cc][0.75]{$X$}
	\psfrag{C}[cc][cc][0.75]{$\coend$}
	\psfrag{g}[cc][cc][0.75]{$\varepsilon$}
	\rsdraw{0.45}{1}{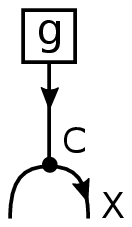}
	\;
	=~~
	\,
	\psfrag{X}[cc][cc][0.75]{$X$}
	\rsdraw{0.45}{1}{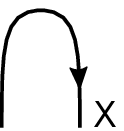}
	\;.
	$$
	
	An important factorization property is given in the following lemma.
	\begin{lem}[Fubini theorem for coends, \cite{catswork}] \label{CoendFubini}  Let~$(\coend,i)$ be a coend of a braided pivotal category~$\mc B$.
		If~$d=\{d_{X_1,\dots, X_n}\co X_1^*\tens X_1\tens \cdots \tens X_n^* \tens X_n \to D\}_{X_1,\dots, X_n \in \Ob{\mc B}}$ is a family
		of morphisms in~$\mc B$, which is dinatural in each~$X_k$ for~$1\leq k \leq n$, then there exists a unique morphism~$\varphi\co \coend^{\tens n} \to D$ in~$\mc B$ such that \[d_{X_1,\dots,X_n}=\varphi(i_{X_1}\tens \cdots \tens i_{X_n})\]
		for all~$X_1,\dots, X_n \in \Ob{\mc B}$.
	\end{lem}
	
	If $\mc B$ is a braided pivotal category, coend $\coend$ of $\mc B$ is a Hopf algebra in $\mc B$ (see \cite{LYUBASHENKO1995279, {MAJID1993187}, {moncatstft}}), which means that the coproduct and the counit are algebra morphisms and that there is an antipode.
	The unit is~$u=(\id_\uu \tens i_{\uu})(\coevl_{\uu} \tens \id_\uu) \co \uu \to \coend$. Multiplication~$m \co \coend \tens \coend \to \coend$ and antipode~$S \co \coend \to \coend$ are unique morphisms such that for all~$X,Y \in \Ob{\mc B}$,
	$$
	\,
	\psfrag{X}[cc][cc][0.75]{$X$}
	\psfrag{Y}[cc][cc][0.75]{$Y$}
	\psfrag{C}[cc][cc][0.75]{$\coend$}
	\psfrag{w}[cc][cc][0.85]{$m$}
	\rsdraw{0.45}{1}{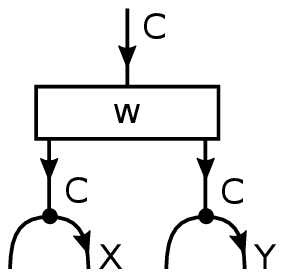}
	\;=~~
	\,
	\psfrag{X}[cc][cc][0.75]{$X$}
	\psfrag{Y}[cc][cc][0.75]{$Y$}
	\psfrag{B}[cc][cc][0.75]{\hspace{0.2cm}$\coend$}
	\psfrag{id}{\scalebox{0.85}{$\id_{Y\tens X}$}}
	\rsdraw{0.45}{1}{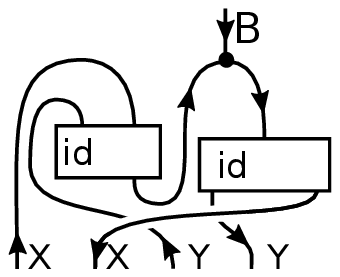}
	\;, \qquad
	\,
	\psfrag{X}[cc][cc][0.75]{$X$}
	\psfrag{C}[cc][cc][0.75]{$\coend$}
	\psfrag{g}[cc][cc][0.75]{$S$}
	\rsdraw{0.45}{1}{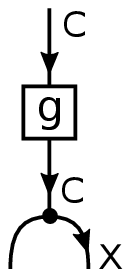}
	\;
	=
	\,
	\psfrag{X}[cc][cc][0.75]{$X$}
	\psfrag{C}[cc][cc][0.75]{$\coend$}
	\rsdraw{0.45}{1}{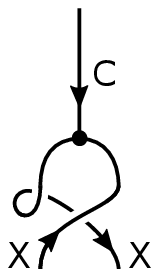}
	\;.
	$$
	Note that $S^2= \theta_\coend^r$.
	
	\subsection{Evaluations  of ribbon string links}\label{hopfdiagforms}
	Let $\mc B$ be a ribbon category with a coend $\coend$.
We recall the con\-struc\-tion of the func\-tor
	$$\phi_{\mc B} \co \RSL \to \text{Conv}_{\mc B}(\coend,\uu)$$ from~\cite{hopfdiag}, which is im\-por\-tant in the se\-quel.
	It is identity on objects. Let $n$ be a non-negative integer.
	For an~$n$-string link~$T$, the morphism~$\phi_{\mc B}(T)\co \coend^{\tens n} \to \uu$ is defined as follows. Let~$i=\{i_X\co X^* \tens X \to \coend\}_{X \in \Ob{\mc B}}$ be the universal dinatural transformation associated to the coend $\coend$.
	First, we orient the ribbon~$n$-handle~$\mathfrak{F}(T)$ as prescribed in Section~\ref{ribbon handles}.
	Coloring the $k$-th ribbon of~$\mathfrak{F}(T)$ by an object $X_k$ of $\mc B$, we obtain a family of morphisms $$\mathfrak{F}(T)_{X_1,\dots, X_n}\co X_1^*\tens X_1\tens \cdots\tens  X_n^*\tens X_n\to \uu,$$ which is dinatural in each variable.
	Hence, it factorizes by Lemma \ref{CoendFubini}:
	
	\begin{equation} \label{coendeval} \mathfrak{F}(T)_{X_1,\dots, X_n}=\phi_\mc{B}(T) \circ (i_{X_1}\tens \cdots \tens i_{X_n})\end{equation}
	for a unique morphism $\phi_\mc{B}(T)\co \coend^{\tens n} \to \uu$.
	Note that~$\mathfrak{F}(T)_{X_1,\dots, X_n}$ is the value of the~$\mc B$-colored (as above) ribbon~$n$-handle $\mathfrak{F}(T)$ under the Reshetikhin-Turaev functor (see~\cite[Theorem 2.5]{turaevqinvariants}) and the morphism~$\phi_{{\mc B}}(T)$ is the universal quantum invariant derived from~$\mc B$ of the ribbon~$n$-string link~$T$.
	
	\subsection{Cyclic sets from string links and quantum invariants} Let $\mc B$ be a ribbon cate\-gory with a coend $\coend$. The coend $\coend$ is a Hopf algebra in $\mc B$ and so gives rise to the cocyclic set $\coend^\bullet$ and the cyclic set $\coend_\bullet$ (see Section \ref{algebraiccyclic}). Recall from Section \ref{slrelatedobjects} the cocyclic set $\mc{SL}^\bullet$ and cyclic set $\mc{SL}_\bullet$ defined geometrically via ribbon string links. Consider the evaluation functor~$\phi_{\mc B} \co \RSL \to \text{Conv}_{\mc B}(\coend, \uu)$ from Section \ref{hopfdiagforms}.  The next theorem relates these (co)cyclic sets via quantum invariants.
	\begin{thm} \label{slmain2} The evaluation functor~$\phi_{\mc B}$ induces a morphism of cocyclic sets from~${\mc {SL}}^\bullet$ to~$\coend^\bullet$ and a morphism of cyclic sets from~${\mc {SL}}_\bullet$ to~$\coend_\bullet$.
	\end{thm}
Theorem \ref{slmain2} says that in a sense, $\mc{SL}^\bullet$ is an initial cocyclic set, which is universal with respect to ribbon categories with a coend. Similarly, $\mc{SL}_\bullet$ is an initial cyclic set, which is universal with respect to ribbon categories with a coend.

	Recall the cyclic (respectively, cocyclic) sets~${\mc {SL}}^\bullet\circ L$ and~$\coend^\bullet \circ L$ (respectively, ${\mc {SL}}_\bullet \circ {L^{\opp}}$ and~$\coend_\bullet \circ L^\opp$) from Sections~\ref{SLL} and~\ref{passing}. An immediate corollary of Theorem~\ref{slmain2} is the following:
	\begin{cor} \label{corslmain2} The evaluation functor~$\phi_{\mc B}$ induces a morphism of cyclic sets from~${\mc {SL}}^\bullet \circ L$ to~$\coend^\bullet \circ L$ and a morphism of cocyclic sets from~${\mc {SL}}_\bullet \circ {L^{\opp}}$ to~$\coend_\bullet \circ L^\opp$.
	\end{cor}

	\subsection{Proof of Theorem \ref{slmain2}} \label{pfslmain2}
We first prove two lemmas.
	
	\begin{lem}\label{lemcyclic} Let $n\in \N^*$. For any $(n+1)$-ribbon handle $H$, the ribbon handles \[\,
		\psfrag{b}[cc][cc][0.75]{$0$}
		\psfrag{bn}[cc][cc][0.75]{\hspace{-0.2cm}$1$}
		\psfrag{er}[cc][cc][0.75]{\hspace{0.53cm}$n-1$}
		\psfrag{e}[cc][cc][0.75]{$n$}
		\psfrag{H}[cc][cc][0.85]{\hspace{0.35cm}$H$}
		\rsdraw{0.55}{1}{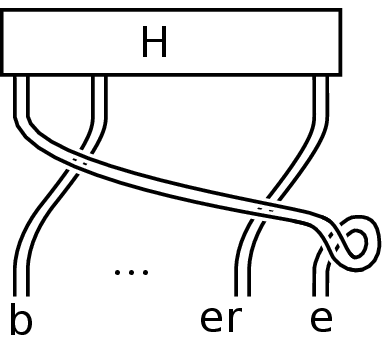}\; \quad \text{and} \quad \hspace{0.35cm} \,
		\psfrag{b}[cc][cc][0.75]{$0$}
		\psfrag{bn}[cc][cc][0.75]{\hspace{-0.2cm}$1$}
		\psfrag{er}[cc][cc][0.75]{\hspace{0.45cm}$n-1$}
		\psfrag{e}[cc][cc][0.75]{$n$}
		\psfrag{H}[cc][cc][0.85]{\hspace{0.35cm}$H$}
		\rsdraw{0.55}{1}{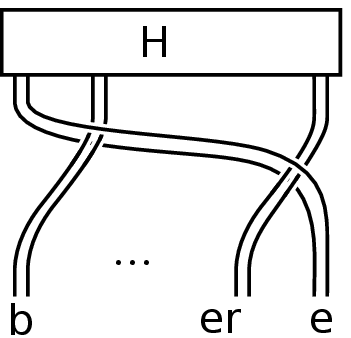}\;\] are isotopic.
	\end{lem}
	
	\begin{proof} By isotopy and Reidemeister moves, we have
		\begingroup
		\allowdisplaybreaks
		\begin{align*}
		&\psfrag{b}[cc][cc][0.75]{$0$}
		\psfrag{bn}[cc][cc][0.75]{\hspace{-0.2cm}$1$}
		\psfrag{er}[cc][cc][0.75]{\hspace{0.53cm}$n-1$}
		\psfrag{e}[cc][cc][0.75]{$n$}
		\psfrag{H}[cc][cc][0.85]{\hspace{0.35cm}$H$}
		\rsdraw{0.55}{1}{cyclicRH}\;=\,
		\psfrag{b}[cc][cc][0.75]{$0$}
		\psfrag{bn}[cc][cc][0.75]{\hspace{-0.2cm}$1$}
		\psfrag{er}[cc][cc][0.75]{\hspace{0.53cm}$n-1$}
		\psfrag{e}[cc][cc][0.75]{$n$}
		\psfrag{H}[cc][cc][0.85]{\hspace{0.35cm}$H$}
		\rsdraw{0.55}{1}{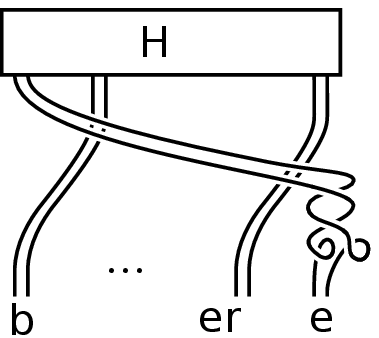}\;=\,
		\psfrag{b}[cc][cc][0.75]{$0$}
		\psfrag{bn}[cc][cc][0.75]{\hspace{-0.2cm}$1$}
		\psfrag{er}[cc][cc][0.75]{\hspace{0.53cm}$n-1$}
		\psfrag{e}[cc][cc][0.75]{$n$}
		\psfrag{H}[cc][cc][0.85]{\hspace{0.35cm}$H$}
		\rsdraw{0.55}{1}{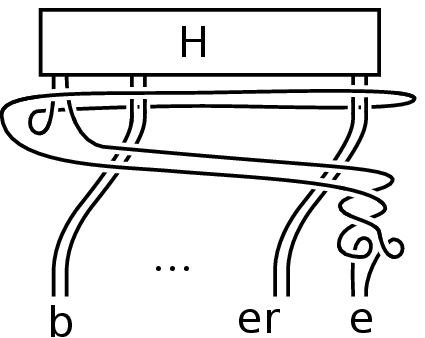}\;\\
		=&\,
		\psfrag{b}[cc][cc][0.75]{$0$}
		\psfrag{bn}[cc][cc][0.75]{\hspace{-0.2cm}$1$}
		\psfrag{er}[cc][cc][0.75]{\hspace{0.53cm}$n-1$}
		\psfrag{e}[cc][cc][0.75]{$n$}
		\psfrag{H}[cc][cc][0.85]{\hspace{0.35cm}$H$}
		\rsdraw{0.55}{1}{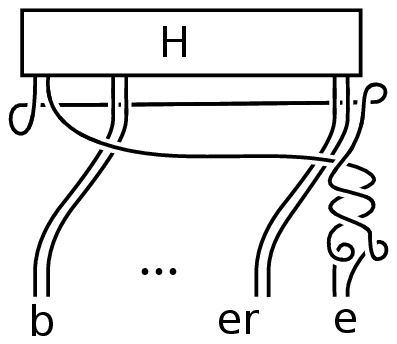}\;=\,
		\psfrag{b}[cc][cc][0.75]{$0$}
		\psfrag{bn}[cc][cc][0.75]{\hspace{-0.2cm}$1$}
		\psfrag{er}[cc][cc][0.75]{\hspace{0.53cm}$n-1$}
		\psfrag{e}[cc][cc][0.75]{$n$}
		\psfrag{H}[cc][cc][0.85]{\hspace{0.35cm}$H$}
		\rsdraw{0.55}{1}{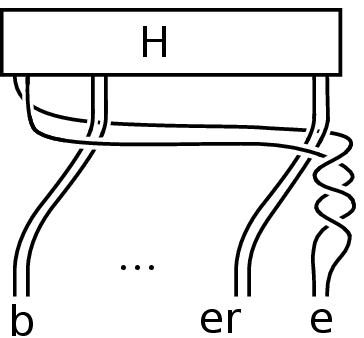}\;	=	
		\,
		\psfrag{b}[cc][cc][0.75]{$0$}
		\psfrag{bn}[cc][cc][0.75]{\hspace{-0.2cm}$1$}
		\psfrag{er}[cc][cc][0.75]{\hspace{0.53cm}$n-1$}
		\psfrag{e}[cc][cc][0.75]{$n$}
		\psfrag{H}[cc][cc][0.85]{\hspace{0.35cm}$H$}
		\rsdraw{0.55}{1}{cyclicRH7}\;.
		\end{align*}
		\endgroup
	\end{proof}

Recall the isomorphism $\mathfrak{F} \co \textbf{RSL} \to \textbf{RH}$ from Section \ref{ribbon handles}.
	
	\begin{lem}\label{lemprepa} Let $n\in  \N^*$.  For any ribbon $(n+1)$-string link $T$, the ribbon handle $\mathfrak{F}(\mc{SL}^\bullet(\tau_n)(T))$ is isotopic to the ribbon handle \[\,
		\psfrag{b}[cc][cc][0.75]{\hspace{-0.1cm}$0$}
		\psfrag{er}[cc][cc][0.75]{\hspace{0.65cm}$n-1$}
		\psfrag{e}[cc][cc][0.75]{\hspace{0.3cm}$n$}
		\psfrag{H}[cc][cc][0.85]{\hspace{1cm}$\mathfrak{F}(T)$}
		\rsdraw{0.45}{1}{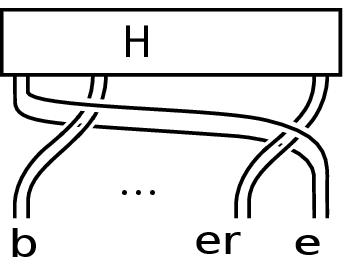}\;
		.\]
	\end{lem}
	
	\begin{proof}
	For any $T\in \mc{RSL}_1$, we have
		\[\mathfrak{F}(\mc{SL}_\bullet(\tau_1)(T))=\,
		\psfrag{b}[cc][cc][0.75]{$0$}
		\psfrag{e}[cc][cc][0.75]{$1$}
		\psfrag{T}[cc][cc][0.85]{$T$}
		\rsdraw{0.45}{1}{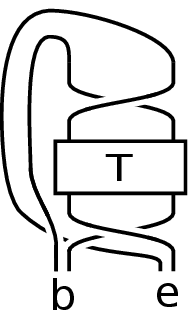}\;=\,
		\psfrag{b}[cc][cc][0.75]{$0$}
		\psfrag{e}[cc][cc][0.75]{$1$}
		\psfrag{T}[cc][cc][0.85]{$T$}
		\rsdraw{0.45}{1}{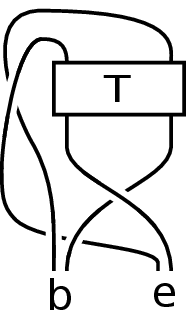}\;=\,
		\psfrag{b}[cc][cc][0.75]{$0$}
		\psfrag{e}[cc][cc][0.75]{$1$}
		\psfrag{T}[cc][cc][0.85]{$T$}
		\rsdraw{0.45}{1}{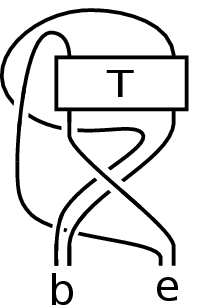}\;=\,
		\psfrag{b}[cc][cc][0.75]{$0$}
		\psfrag{e}[cc][cc][0.75]{\hspace{0.3cm}$1$}
		\psfrag{H}[cc][cc][0.85]{\hspace{0.75cm}$\mathfrak{F}(T)$}
		\rsdraw{0.45}{1}{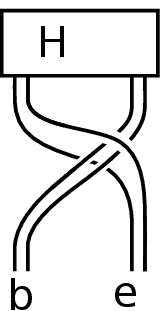}\;.\]
	This proves the lemma for $n=1$. The general case is similar and is left to the reader.	
		
	\end{proof}

Let us prove Theorem \ref{slmain2}.
Since~$\phi_\mc{B}$ is a functor, it induces for any~$n\in \N$ the map~\[\phi^n \co \End_{\RSL}(n+1) \to \End_{\text{Conv}_{\mc B}(\coend,\uu)}(n+1).\]
This defines a family of maps~$\phi^\bullet=\{\phi^n\co \mc{RSL}_n \to \Hom_{\mc B}(\coend^{\tens n+1}, \uu)\}_{n\in \N}.$
To show that~$\phi^\bullet$ defines a morphism of cocyclic sets from~$\mc{SL}^\bullet$ to~$\coend^\bullet$, we have to verify that
\begin{align}
\label{commfacsl}\phi^n(\delta_i^n(T))&= \delta_i^n (\phi^{n-1}(T)) \quad \text{for all}  \quad n\ge 1, 0\le i \le n, \quad \text{and} \quad T\in {\mc {RSL}}_{n-1},  \\
\label{commdegsl}\phi^n(\sigma_j^n(T))&= \sigma_j^n (\phi^{n+1}(T)) \quad \text{for all} \quad n\ge 0, 0\le j \le n,  \quad \text{and} \quad T\in {\mc {RSL}}_{n+1}, \\
\label{commcycsl}\phi^n(\tau_n(T)) &= \tau_n(\phi^n(T))  \quad \text{for all} \quad n\ge 0 \quad \text{and} \quad T\in {\mc{RSL}}_n.
\end{align}
By abuse, we use here the same notation for cofaces, codegeneracies, and cocyclic operators of~$\mc{SL}^\bullet$ and~$\coend^\bullet$.
	
	Let us prove~\eqref{commfacsl}. Assume that $1 \le i \le n-1$ and~$T \in \mc{RSL}_{n-1}$. We have:
	\[
	\mathfrak{F}(\delta_i^n(T))=\mathfrak{F}\left(
	\,
	\psfrag{b}[cc][cc][0.75]{$1$}
	\psfrag{i}[cc][cc][0.75]{\hspace{0.3cm}$i$}
	\psfrag{i+1}[cc][cc][0.75]{\hspace{0.4cm}$i+1$}
	\psfrag{e}[cc][cc][0.75]{$n$}
	\psfrag{T}[cc][cc][0.85]{$T$}
	\rsdraw{0.55}{1}{SLdeltai}\;
	\right)=\,
	\psfrag{b}[cc][cc][0.75]{$1$}
	\psfrag{i}[cc][cc][0.75]{\hspace{0.2cm}$i$}
	\psfrag{i+1}[cc][cc][0.75]{\hspace{0.4cm}$i+1$}
	\psfrag{e}[cc][cc][0.75]{$n$}
	\psfrag{T}[cc][cc][0.85]{\hspace{0.9cm}$\mathfrak{F}(T)$}
	\rsdraw{0.55}{1}{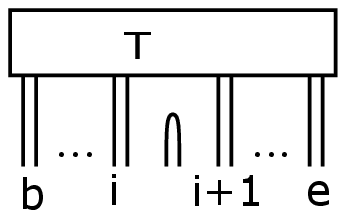}\;.
	\]
	Consequently, by equation~\eqref{coendeval} and by definition of the counit of coend of $\mc B$ (see Section~\ref{coendprelimini}),
	\[\phi^n(\delta_i^n(T))=\,
	\psfrag{b}[cc][cc][0.75]{$1$}
	\psfrag{i}[cc][cc][0.75]{\hspace{0.2cm}$i$}
	\psfrag{i+1}[cc][cc][0.75]{\hspace{0.4cm}$i+1$}
	\psfrag{e}[cc][cc][0.75]{$n$}
	\psfrag{T}[cc][cc][0.85]{\hspace{0.8cm}$\phi^{n-1}(T)$}
	\rsdraw{0.55}{1}{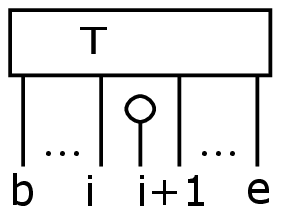}\;=\delta_i^n(\phi^{n-1}(T)).\]  The cases when $n\ge 1$ and $i=0$ or $i=n$ are verified analogously.

	Next, let us prove \eqref{commdegsl}.
	Let $n\ge 0, 0\le j \le n$ and $T\in \mc{RSL}_{n+1}$.
	 We have:
	\[\mathfrak{F}(\sigma_j^n(T))=\mathfrak{F}\left(\,
	\psfrag{b}[cc][cc][0.75]{$0$}
	\psfrag{j}[cc][cc][0.75]{\hspace{0.45cm}$j$}
	\psfrag{j+1}[cc][cc][0.75]{$j+1$}
	\psfrag{e}[cc][cc][0.75]{$n+1$}
	\psfrag{T}[cc][cc][0.85]{$T$}
	\rsdraw{0.45}{1}{SLsigmajv2}\;\right)
	\;=\,
	\psfrag{b}[cc][cc][0.75]{$0$}
	\psfrag{j}[cc][cc][0.75]{$j$}
	\psfrag{e}[cc][cc][0.75]{\hspace{0.6cm}$n+1$}
	\psfrag{T}[cc][cc][0.85]{$T$}
	\rsdraw{0.5}{1}{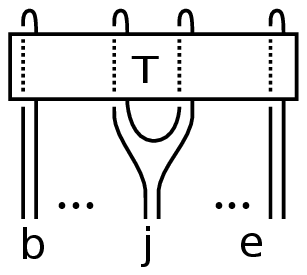}=\,
	\psfrag{b}[cc][cc][0.75]{$0$}
	\psfrag{j}[cc][cc][0.75]{\hspace{0.2cm}$j$}
	\psfrag{e}[cc][cc][0.75]{$n+1$}
	\psfrag{T}[cc][cc][0.85]{\hspace{0.35cm}$\mathfrak{F}(T)$}
	\rsdraw{0.5}{1}{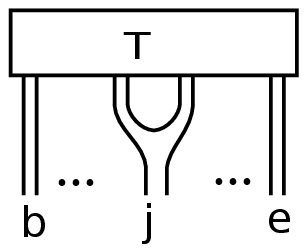}
	\;.
	\]
	Consequently, by equation~\eqref{coendeval} and by definition of the comultiplication of coend of~$\mc B$ (see Section~\ref{coendprelimini}),
	\[\phi^n(\sigma_j^n(T))=
	\,
	\psfrag{b}[cc][cc][0.75]{\hspace{-0.1cm}$0$}
	\psfrag{j}[cc][cc][0.75]{$j$}
	\psfrag{e}[cc][cc][0.75]{\hspace{-0.1cm}$n+1$}
	\psfrag{T}[cc][cc][0.85]{\hspace{1.5cm}$\phi^{n+1}(T)$}
	\rsdraw{0.55}{1}{HDsj}\;=\sigma_j^n(\phi^{n+1}(T)).\]
	
	Finally, let us prove~\eqref{commcycsl}. Let~$n\ge 0$ and~$T\in \mc{RSL}_n$. If~$n=0$, then by definition, we have for any~$1$-string link~$T$,~$\phi^0(\tau_0(T))= \phi^0(T) = \tau_0(\phi^0(T)).$
	Assume that~$n\ge 1$. We have
	\[\mathfrak{F}(\tau_n(T))\overset{(i)}{=}\,
	\psfrag{b}[cc][cc][0.75]{$0$}
	\psfrag{bn}[cc][cc][0.75]{\hspace{-0.3cm}$1$}
	\psfrag{er}[cc][cc][0.75]{\hspace{0.35cm}$n-1$}
	\psfrag{e}[cc][cc][0.75]{\hspace{0.35cm}$n$}
	\psfrag{H}[cc][cc][0.85]{\hspace{0.85cm}$\mathfrak{F}(T)$}
	\rsdraw{0.45}{1}{cyclicRH7v2}\;\overset{(ii)}{=}\,
	\psfrag{b}[cc][cc][0.75]{$0$}
	\psfrag{bn}[cc][cc][0.75]{\hspace{-0.3cm}$1$}
	\psfrag{er}[cc][cc][0.75]{\hspace{0.35cm}$n-1$}
	\psfrag{e}[cc][cc][0.75]{$n$}
	\psfrag{H}[cc][cc][0.85]{\hspace{0.55cm}$\mathfrak{F}(T)$}
	\rsdraw{0.45}{1}{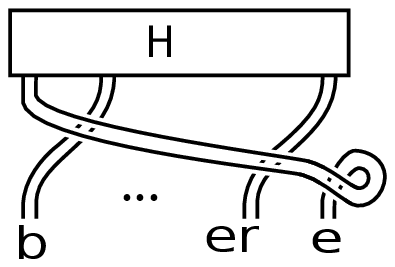}\;.\]
	Here $(i)$ follows by Lemma \ref{lemprepa} and $(ii)$ follows by Lemma \ref{lemcyclic}.
	Consequently,  by equation \eqref{coendeval} and by naturality of the twist, we have
	\[\phi^n(\tau_n(T))=\,
	\psfrag{b}[cc][cc][0.75]{$0$}
	\psfrag{bn}[cc][cc][0.75]{\hspace{-0.4cm}$1$}
	\psfrag{er}[cc][cc][0.75]{\hspace{0.35cm}$n-1$}
	\psfrag{e}[cc][cc][0.75]{$n$}
	\psfrag{T}[cc][cc][0.85]{$\phi^n(T)$}
	\rsdraw{0.45}{1}{HDtn}\;=\tau_n(\phi^n(T)).\]
	This completes the proof of the fact that the evaluation~$\phi_\mc{B}$ induces a morphism of cocyclic sets from~$\mc{SL}^\bullet$ to~$\coend^\bullet$.
	
	Let us prove that that~$\phi_\mc{B}$ induces a morphism between cyclic sets~$\mc{SL}_\bullet$ and~$\coend_\bullet$.
	As above, the functor~$\phi_{\mc B}\co \RSL \to \text{Conv}_{\mc B}(\coend,\uu)$ induces for any $n\in \N$ the map
	$$\phi_n\co \End_{\RSL}(n+1) \to \End_{\text{Conv}_{\mc B}(\coend,\uu)} (n+1).$$
This defines a family of maps~$\phi_{\bullet}=\{\phi_{n} \co \mc {RSL}_n \to \Hom_{\mc B}(\coend^{\tens n+1}, \uu)\}_{n\in \N}$. Recall that for any~$n\in \N$,
	$$\mc {SL}_n=\mc {RSL}_n \quad \text{and} \quad \coend_n=\coend^n=\Hom_{\mc B}(\coend^{\tens n+1}, \uu) \quad \text{as sets}.$$
	To show that $\phi_{\bullet}$ defines a morphism of cyclic sets from $\mc{SL}_\bullet$ to $\coend_\bullet$, we have to verify that
	\begin{align}
	\label{commfacsl2}\phi_{n-1}(d_i^n(T))&= d_i^n (\phi_{n}(T)) \quad \text{for all} \quad n\ge 1, 0\le i \le n,  \quad \text{and} \quad T\in {\mc {RSL}}_{n},  \\
	\label{commdegsl2}\phi_{n+1}(s_j^n(T))&= s_j^n (\phi_{n}(T)) \quad \text{for all} \quad n\ge 0, 0\le j \le n,  \quad \text{and} \quad T\in {\mc {RSL}}_{n}, \\
	\label{commcycsl2}\phi_n(t_n(T)) &= t_n(\phi_n(T))  \quad \text{for all} \quad n\ge 0 \quad \text{and} \quad T\in {\mc{RSL}}_n.
	\end{align}
	
	To prove \eqref{commfacsl2} and \eqref{commdegsl2}, we use the fact that the coend $\coend$ of $\mc B$ coacts on each object $X$ of $\mc B$ via the universal coaction $\delta_X\co X\to X\tens \coend$, defined by $\delta_X= (\id_X\tens i_X)(\coevl_X \tens \id_X)$. We depict $\delta_X$ by
	$$\delta_X = \,
	\psfrag{X}[cc][cc][0.75]{$X$}
	\psfrag{C}[cc][cc][0.75]{$\coend$}
	\rsdraw{0.45}{1}{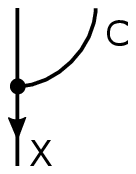} \; , \quad  \text{ so by definition} \quad \,
	\psfrag{X}[cc][cc][0.75]{$X$}
	\psfrag{C}[cc][cc][0.75]{$\coend$}
	\rsdraw{0.45}{1}{unicoact}\; = \,
	\psfrag{X}[cc][cc][0.75]{$X$}
	\psfrag{C}[cc][cc][0.75]{\hspace{-0.1cm}$\coend$}
	\rsdraw{0.45}{1}{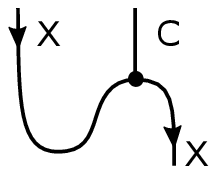}\;.$$
	Next, by definition of unit~$u$ of coend~$\coend$ (see Section \ref{coendprelimini}), we have $u=\delta_\uu$. Similarly, the multiplication $m$ of coend $\coend$ is characterized by the following universal property: for all~$X,Y \in \Ob{\mc B}$,
	$$\delta_X = \,
	\psfrag{X}[cc][cc][0.75]{$X$}
		\psfrag{Y}[cc][cc][0.75]{$Y$}
		\psfrag{g}[cc][cc][0.85]{\hspace{0.1cm}$m$}
	\psfrag{C}[cc][cc][0.75]{$\coend$}
	\rsdraw{0.45}{1}{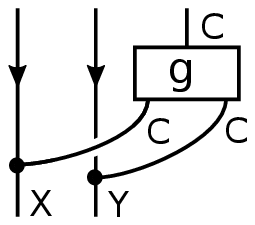} \hspace{0.1cm} = \,
	\psfrag{X}{$X$}
		\psfrag{XY}[cc][cc][0.75]{\hspace{0.55cm}$X\tens Y$}
	\psfrag{C}[cc][cc][0.75]{$\coend$}
	\rsdraw{0.45}{1}{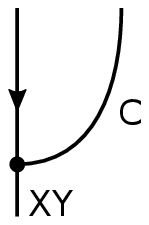} \hspace{0.2cm}.$$

Let us express the universal isotopy invariant $\phi_{B}(T)$ of an $n$-string link $T$ in terms of the universal coaction. By coloring the $k$-th ribbon of $T$ by an object $X_k$ of $\mc B$, we obtain a family of morphisms $$T_{X_1,\cdots,X_n} \co X_1 \tens  \cdots \tens X_n \to X_1 \tens  \cdots \tens X_n,$$ which is natural in each variable. It factorizes as follows:
	\begin{equation}
	T_{X_1,\dots,X_n}\overset{(i)}{=} \,
	\psfrag{m}[cc][cc][0.75]{$X_1$}
	\psfrag{n}[cc][cc][0.75]{$X_2$}
	\psfrag{e}[cc][cc][0.75]{\hspace{-0.1cm}$X_n$}
	\psfrag{T}[cc][cc][0.85]{$\mathfrak{F}(T)_{X_1,\dots,X_n}$}
	\rsdraw{0.45}{1}{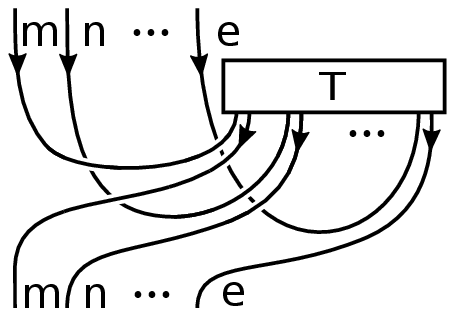}\;\overset{(ii)}{=} \,
	\psfrag{m}[cc][cc][0.75]{$X_1$}
	\psfrag{n}[cc][cc][0.75]{$X_2$}
	\psfrag{e}[cc][cc][0.75]{\hspace{-0.1cm}$X_n$}
	\psfrag{T}[cc][cc][0.85]{$\phi_{\mc B}(T)$}
	\rsdraw{0.45}{1}{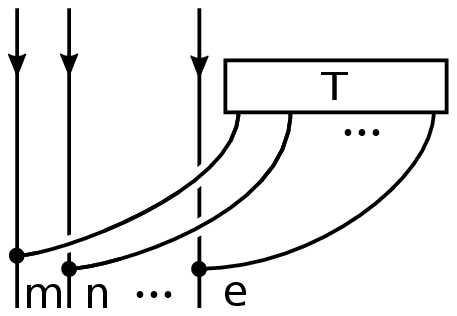}\;.
	\end{equation}
Here $(i)$ follows by definition of mutually inverse functors $\mathfrak{F}$ and $\mathfrak{G}$ (see Section \ref{ribbon handles}) and by isotopy invariance of graphical calculus, $(ii)$ follows by equation \eqref{coendeval}, naturality of the braiding, and definition of the universal coaction.
	
	Let us prove~\eqref{commfacsl2}. Assume that~$1\le i \le n-1$, and~$T \in \mc {RSL}_n$. For all objects~$X_0, \dots, X_{i-1}, X_{i+1}, \dots, X_{n}\in \Ob{\mc B}$,
	$$\,
	\psfrag{m}[cc][cc][0.75]{$X_0$}
	\psfrag{o}[cc][cc][0.75]{{$X_{i+1}$}}
	\psfrag{n}[cc][cc][0.75]{\hspace{0.45cm}{$X_{i-1}$}}
	\psfrag{e}[cc][cc][0.75]{\hspace{0.1cm}$X_{n}$}
	\psfrag{T}[cc][cc][0.85]{$\phi_{n}(T)$}
	\rsdraw{0.45}{1}{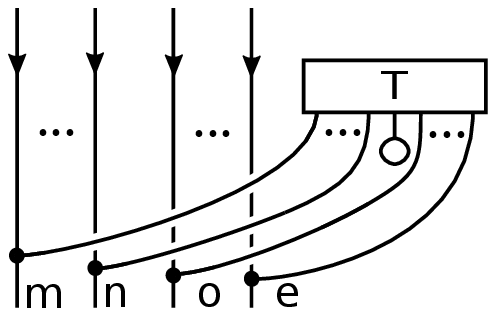}\;= \,
	\psfrag{m}[cc][cc][0.75]{$X_1$}
	\psfrag{n}[cc][cc][0.75]{\hspace{0.45cm}{$X_{i-1}$}}
	\psfrag{o}[cc][cc][0.75]{{$X_{i+1}$}}
	\psfrag{1}[cc][cc][0.75]{$\uu$}
	\psfrag{e}[cc][cc][0.75]{$X_{n}$}
	\psfrag{T}[cc][cc][0.85]{$\phi_{n}(T)$}
	\rsdraw{0.45}{1}{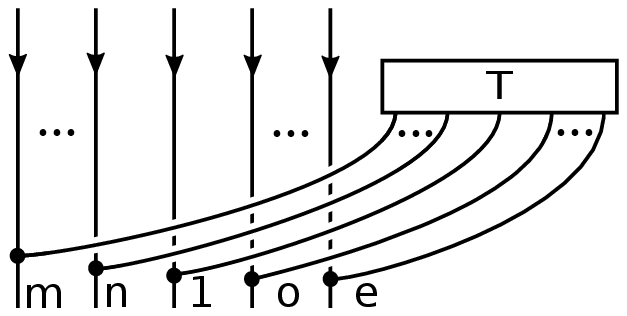}\;.$$
	Hence
	$$\phi_{n-1}(d_i^n(T)) = d_i^n(\phi_n(T)).$$
	The cases when $n\ge 1$ and $i=0$ or $i=n$ are verified analogously.
	
	Let us prove \eqref{commdegsl2}. Let $n\ge 0$, $0\le j \le n$, and $T \in \mc {RSL}_n$. For all~$X_0, \dots, X_{n+1}\in \Ob{\mc B}$,
	$$\,
	\psfrag{m}[cc][cc][0.75]{$X_0$}
	\psfrag{n}[cc][cc][0.75]{{$X_j$}}
	\psfrag{o}[cc][cc][0.75]{{$X_{j+1}$}}
	\psfrag{e}[cc][cc][0.75]{$X_{n+1}$}
	\psfrag{T}[cc][cc][0.85]{$\phi_{n}(T)$}
	\rsdraw{0.45}{1}{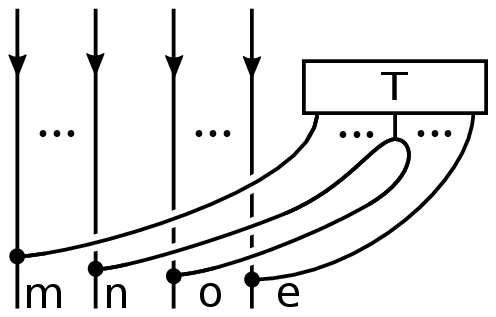}\; = \,
	\psfrag{m}[cc][cc][0.75]{$X_0$}
	\psfrag{n}[cc][cc][0.75]{\hspace{0.5cm}\scalebox{0.6}{$X_j\tens X_{j+1}$}}
	\psfrag{e}[cc][cc][0.75]{$X_{n+1}$}
	\psfrag{T}[cc][cc][0.85]{$\phi_n(T)$}
	\rsdraw{0.45}{1}{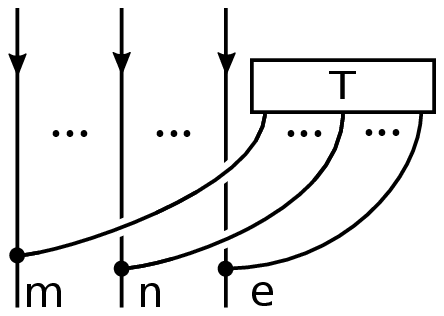}\;$$
	Hence
	$$\phi_{n+1}(s_j^n(T)) = s_j^n(\phi_n(T)).$$
		
	Note that~\eqref{commcycsl2} follows from the equation~\eqref{commcycsl} combined with the fact that the cyclic operators of $\mc {SL}_\bullet$ are inverse to cocyclic operators of $\mc {SL}^\bullet$ and the fact that for any $n\in \N$, $\phi_n = \phi^n$ as functions.
	This completes the proof of the fact that the evaluation~$\phi_\mc{B}$ induces a morphism of cyclic sets from~$\mc{SL}_\bullet$ to~$\coend_\bullet$.

	\bibliographystyle{abbrv}
	\bibliography{biblionart2}
	
\end{document}